\documentclass[reqno,11pt]{amsart}
\usepackage[left=1in,right=1in]{geometry}

\usepackage{titlesec}
 \titleformat {\section}
   {\large \sc \centering}{\thesection}{1em}{}
 \titleformat {\subsection}
   {\normalsize \bfseries \raggedright}{\thesubsection}{1em}{}

\usepackage{amssymb}
\usepackage{mathtools} 

\usepackage[utf8]{inputenc}
\usepackage[T1]{fontenc}
\usepackage[mathscr]{eucal}

\PassOptionsToPackage{
  backend=bibtex8,
  bibencoding=ascii,
  language=auto,
  style=numeric-comp,%
  giveninits=true,
  sorting=nyt, 
  maxbibnames=10, 
  }{biblatex}
\usepackage{biblatex}

\usepackage{paralist}

\usepackage[dvipsnames,svgnames]{xcolor}
\colorlet{MyBlue}{DodgerBlue!60!Black}

\usepackage[font=small,labelfont=bf]{caption}
\usepackage{subcaption}

\usepackage{hyperref}
\hypersetup{
	colorlinks=true,
	linktocpage=true,
	pdfstartview=FitH,
	breaklinks=true,
	pdfpagemode=UseNone,
	pageanchor=true,
	pdfpagemode=UseOutlines,
	plainpages=false,
	bookmarksnumbered,
	bookmarksopen=false,
	bookmarksopenlevel=1,
	hypertexnames=true,
	pdfhighlight=/O,
	urlcolor=MyBlue!60!black,
	linkcolor=MyBlue!70!black,
	citecolor=DarkGreen!70!black, 
	pdftitle={},
	pdfauthor={},
	pdfsubject={},
	pdfkeywords={},
	pdfcreator={pdfLaTeX},
	pdfproducer={LaTeX with hyperref}
}

\numberwithin{equation}{section}  
\usepackage[sort&compress,capitalize,nameinlink]{cleveref}

\crefrangeformat{equation}{\upshape(#3#1#4)\textendash(#5#2#6)}

\newtheorem{theorem}{Theorem}
\newtheorem{lemma}{Lemma}
\newtheorem{proposition}{Proposition}
\newtheorem{definition}{Definition}
\newtheorem{corollary}{Corollary}
\theoremstyle{remark}
\newtheorem{remark}{Remark}

\def \E {\mathbb{E}}
\def \N  {\mathbb{N}} 
\def \P {\mathbb{P}}
\def \R  {\mathbb{R}} 
\def \Z  {\mathbb{Z}}

\newcommand{\bigO}{\ensuremath{\mathcal{O}}} 
\newcommand{\smallO}{\ensuremath{o}}         
\newcommand{\smallOmega}{\ensuremath{\omega}}     
\newcommand{\bigTheta}{\ensuremath{\Theta}}     

\newcommand{\correlation}[4][\empty]
{\ensuremath{U^{\ifx#1\empty\relax\else (#1)\fi}_{#2}(#3,#4)}} 

\newcommand{\NNM}{\ensuremath{\nu}} 
\renewcommand{\c}{\ensuremath{\colon}} 

\DeclarePairedDelimiter\cbrac\{\}
\DeclarePairedDelimiter\paren()
\DeclarePairedDelimiter\abs{\lvert}{\rvert}
\DeclarePairedDelimiter{\ceil}\lceil\rceil

\begin{document}
\title[Loop-erased partitioning]{Loop-erased partitioning of a network: \\ monotonicities \& analysis of cycle-free graphs}

\author[L.~Avena]{Luca Avena$^\ddag$} 
\address{{$^\ddag$ Leiden University, Mathematical Institute, Niels Bohrweg 1
		2333 CA, Leiden. The Netherlands.}}
\email{l.avena@math.leidenuniv.nl}

\author[J.E.P.~Driessen]{Jannetje Driessen$^\S$} 
\address{{$^\S$ Leiden University, Mathematical Institute, Niels Bohrweg 1
		2333 CA, Leiden. The Netherlands.}}
\email{jannetjedriessen@gmail.com}

\author[V.T.~Koperberg]{Twan Koperberg$^\star$} 
\address{{$^\star$ Leiden University, Mathematical Institute, Niels Bohrweg 1
		2333 CA, Leiden. The Netherlands.}}
\email{v.t.koperberg@math.leidenuniv.nl}

\subjclass[2020]{05A18, 05C05, 05C81, 05C85, 60J10, 60J28}
\keywords{graph Laplacian, random partitions, loop-erased random walk, rooted spanning forests, determinantal processes}

\begin{abstract} 
We consider random partitions of the vertex set of a given finite graph that can be sampled by means of loop-erased random walks stopped at a random independent exponential time of parameter $q>0$. The related random blocks tend to cluster nodes visited by the random walk associated to the graph on time scale $1/q$. This random partitioning is induced by a measure of rooted spanning forest of the graph which generalizes the classical uniform spanning tree measure and which can be obtained as a zero-limit of FK-percolation with an external cemetery state. Some general properties of this rooted forest measure and related determinantal observables, along with a number of applications in data analysis have been recently explored. We are here mainly interested in the structure the emergent partitioning, referred to as loop-erased partitioning, as the scale parameter $q$ varies.

We first present two main general results shedding light on subtle monoticities properties in $q$ of these rooted forest and associated loop-erased partitioning measures. The first theorem characterizes monotone events in $q$ by deriving a Russo-like formula.
Our second general result concerns two-point correlations defined by the probability that two vertices do not belong to the same block of the partitioning. It states that, on undirected graphs, 
these pairwise-correlation functions are increasing in $q$.  
We then explore other types of results aiming at understanding the emerging asymptotic clusters on simple insightful growing graph models, as $q$ scales with the graph size. Some first results in this direction have been investigated in the recent~\cite{avena2019meanfield} on dense geometries. Here instead we look at very sparse sequences of graphs. We offer a detailed analysis of the resulting partitioning on line segments and we look at simple trees and other almost tree-like geometries, without and with implanted modular structures. For the latter, we characterize emergence of giants and asymptotic detection of these implanted modules.
\end{abstract}

\maketitle

\medskip

\footnotesize

\newpage

\section{Intro: Rooted spanning forests and loop-erased partitioning}\label{intro}

Consider an arbitrary directed weighted finite graph $G=(V, E, w)$ on $n=|V|$ vertices where $E\subseteq\{e=(x,y): x,y \in V \}$ stands for the edge set and $w: E \rightarrow [0,\infty)$ is a given edge-weight function.
We call Random Walk (RW) associated to $G$ the continuous-time Markov chain $X=(X_t)_{t\ge 0}$ with state space $V$ and the \emph{discrete Laplacian} as infinitesimal generator, i.e. the $n\times n$ matrix:
\begin{equation}\label{Laplacian} L=A-D, \end{equation} 
where for any $x,y\in [n]:=\{1,2,\ldots ,n\}$,  $A(x,y)=w(x,y)\mathbf{1}_{\{ x\neq y\}}$ is the \emph{weighted adjacency matrix} and $D(x,y)=\mathbf{1}_{\{ x=y\}}\sum_{z\in[n]\setminus\{x\}} w(x,z) $ is the diagonal matrix guaranteeing that the entries of each row in $L$ sum up to $0$.

Let $\mathcal{F}$ denote the space of rooted spanning forests of $G$, where a rooted spanning forest $F\in\mathcal{F}$ of a graph is a collection of vertex disjoint rooted trees spanning its vertex set. We consider a rooted tree to be a collection of directed edges pointing towards the root. That is, a rooted forest $F$ is a subset of $E$ such that: 
\begin{enumerate}[$(i)$]
  \item each vertex has at most one outgoing edge in $F$;
  \item if there exists a directed path in $F$ from vertex $x$ to vertex $y$, then no such path exists from $y$ to $x$.
\end{enumerate}
The \emph{roots} of $F$ are those vertices without an outgoing edge.

\begin{definition}[{\bf Rooted Spanning Forest of intensity $q$}]\label{def:RSF}Fix a positive parameter $q>0$	and let $\Phi_q$ be the random variable with values in $\mathcal F$ with law:
\begin{equation}\label{RSF}
  \mathbb{P}(\Phi_q=F)=  \frac{q^{r(F)} w(F)}{Z(q)}, \quad\quad F\in \mathcal F, 
\end{equation}
where $w(F):=\prod_{e\in F} w(e)$ stands for the forest weight, $r(F)$ denotes the number of trees (or equivalently the number of roots) in $F\in \mathcal F$, and $Z(q)$ is a normalizing constant referred to as the partition function.
We will refer to this measure as \emph{random rooted spanning forest} of intensity $q$.
\end{definition} In the unitary weight case $w\equiv 1$, when $q=1$, this measure becomes uniform over the set of rooted spanning forests $\mathcal F$ and its structure has been partially analyzed in several geometrical setups in relation to random combinatorial models in statistical physics and coalescence theory, see~\cite{pitman1999coalescent,jones1999weights,pitman2006csp,chebotarev2008goldenratio,kenyon2011laplacian,kenyon2019determinantal, jarai2015avalanche}. 
For any $q > 0$, $\Phi_q$ induces a randomized decomposition of a given network into blocks (corresponding to its trees) and for each block it identifies a representative node (the root of a tree). The presence of the tuning parameter $q$ makes this object natural for exploring a network architecture in a multiscale fashion.
The goal of this paper is to understand the structure of the resulting unrooted random blocks on the set of partitions $\mathcal{P}(V)$ of the vertex set $V$ as the scaling parameter $q$ varies. We refer to this object, defined next, as the Loop-Erased Partitioning (LEP). Its analysis has been initiated on dense graphs in the recent~\cite{avena2019meanfield}. In this work we derive general results on the monotonicity properties of this measure (see Theorems \ref{thm:derivative_probabilities} and \ref{thm:monotone_correlations_symmetric}) and then, by means of these and other properties, we perform a systematic analysis of the emergent partition on very sparse topologies. 

\begin{definition}[{\bf Loop-Erased Partitioning (LEP) of intensity $q$}]\label{LEP}
Given $G=(V, E, w)$, fix a positive parameter $q>0$. We call \emph{loop-erased partitioning of intensity $q$} of the graph, the random unrooted partition, denoted by $\Pi_q$, of $V$, with law:
\begin{equation}\label{LEPmeas}
  \mathbb{P}(\Pi_q=\pi_m)=  \frac{q^m \times \sum_{F\in\mathcal F: \pi(F)=\pi_m } w(F)}{Z(q)}, \quad\quad \pi_m\in \mathcal P(V), \ m\leq |V|, 
\end{equation}
where the sum runs over the space of rooted spanning forests $\mathcal F$ of $G$ and
$\pi(F)$ stands for the partition of $V$ induced by a given rooted spanning forest $F$ where each block is determined by vertices belonging to the same tree, and $m$ counts the number of blocks in the partition $\pi_m$. Equivalently, 
\begin{equation}
  \Pi_q:=\pi(\Phi_q).
\end{equation}
\end{definition}

\subsubsection*{\bf Rooted forest measure and relation to uniform spanning tree:}
The rooted forest $\Phi_q$ is a natural extension of the classical UST (Uniform Spanning Tree) measure which is readily recovered in the constant weight case $w\equiv 1$ by taking the limit of $q$ going to zero in Eq.~\eqref{RSF}. Alternatively, this rooted forest $\Phi_q$ can also be seen as a measure on weighted spanning trees on the extended weighted graph obtained by adding an extra cemetery state accessible from any vertex via an edge with weight $q$. Under this perspective, it is clear that most results known for the UST do have a generalized analogue in the context of this rooted generalized measure. For example, edges in $\Phi_q$ form a determinantal process~\cite{avena2018applications} due to a version of the so-called 
transfer-current theorem~\cite{burton1993transfer}, clarifying its status within negatively associated systems, see~\cite{pemantle2000towards,grimmett2004negative,kahn2010correlation}. Due to the Kirchhoff's matrix tree theorem, the normalizing constant in Eq.~\eqref{LEPmeas} can be expressed as the characteristic polynomial of the matrix $L$ evaluated at $q$, i.e.
\begin{equation}\label{Z}Z(q):=\sum_{F\in \mathcal F }q^{r(F)} w(F)=\det[qI-L],\end{equation}
see e.g.~\cite{avena2018applications, chebotarev1997mft}.
As far as sampling is concerned, for fixed $q>0$, one can use the celebrated algorithm due to Wilson~\cite{wilson1996generating} based on loop-erased random walks. The latter is in fact a classical efficient procedure allowing to sample a rooted tree of a graph with probability proportional to its weight. 
Further, it is well known that the UST can be obtained from the unifying FK-percolation ``super-model'' by properly taking the related interaction parameter to zero, see e.g.~\cite{grimmett2006cluster}. Not surprisingly, as expressed in Lemma~\ref{FK} below, which for simplicity we state in the unitary weight case $w\equiv 1$, the rooted forest in Eq. \eqref{RSF} can also be obtained via a similar zero-limit but by considering a proper FK-percolation with an additional cemetery state. The proof of this proposition is as in~\cite{grimmett2006cluster}, see Thm. 1.23 in Sect 1.5 therein, with the parameters of the FK as specified in the statement below. 

\begin{lemma}[{\bf Rooted forest as zero-limit of extended FK-percolation}]\label{FK}
Given an undirected simple graph $G=(V, E)$, let $G_{\dagger}:=(V_\dagger, E_\dagger)$ be the extended
graph with $V_\dagger:=V\cup\{\dagger\}$ where $\{\dagger\}$ denotes an extra state, 
$E_\dagger:=E\cup\bar{E}$ with $\bar{E}:=\{ (x,\dagger): x\in V\}$. 
Consider the generalized FK-percolation on $G_\dagger$ with parameter $\lambda>0$ and vector of weights $\vec{p}= (p_e)_{e\in E_\dagger}$ such that $p_e=p\in (0,1)$ if $e\in E$ and $p_e= \gamma>0$ for $e\in\bar{E}$, 
that is, the following measure on spanning subgraphs of $G_{\dagger}$ seen as collection of edges in $\Omega:=\{0,1\}^{E_\dagger}$  :
\begin{equation}\label{FKs}
  \mathbb{P}(FK=\omega)=  \frac{\lambda^{k(\omega)} \times \prod_{e\in E} p^{\omega(e)}(1-p)^{1-\omega(e)}
  \prod_{e\in \bar{E}} \gamma^{\omega(e)}(1-\gamma)^{1-\omega(e)}}{Z(\lambda, \vec{p})}, \quad\quad \omega \in\Omega,
 \end{equation}
with $k(\omega)$ counting the number of connected components of the graph $\omega$ and $Z(\lambda, \vec{p})$ being a normalizing constant.
 Assume that $\vec{p}$ is a function of $\lambda$ such that, as $\lambda\rightarrow 0$,
$\gamma=\gamma(\lambda)\rightarrow 0$,  $p=p(\lambda)\rightarrow 0$ and 
$\gamma(\lambda)/p(\lambda)\rightarrow q\in (0,\infty)$.
Then as $\lambda$ goes to zero, the law in Eq.~\eqref{FKs} (projected onto subgraphs of $G$) degenerates into the law of the random rooted forest $\Phi_q$ in Eq.~\eqref{RSF} with unitary weights.
\end{lemma}

Yet, if the UST can be seen as the ``static global random backbone'' of a given network, the forest process $(\Phi_q)_{q>0}$ represents its ``mesoscopic and dynamic'' analogue where the notion of locality is captured parametrically by what the RW sees on time-scale $1/q$. 
As such, it naturally leads to dynamic multi-scale approaches(see~\cite[Thm.2]{avena2018applications}), and new structures and questions which do not make sense within the more restrictive global and static UST context.

\subsubsection*{\bf Applications of rooted forest measure and LEP:}
In a series of recent works~\cite{avena2018applications,avena2018transfercurrent,avena2018randomforests} some general properties of the rooted forest mesure have been explored. For example, the roots~\cite[Prop.2.2]{avena2018applications} in $\Phi_q$ form a determinantal point process with kernel given by the RW Green's function, that is: for any $A\subset V$
\begin{equation}\label{DetR}
  \mathbb{P}( A \text{ is in the set of roots})= \det[K_q]_A,
\end{equation}
with $\left[K_{q}\right]_{A}$ being the restriction of the matrix $K_q:= q(qI-L)^{-1}$ to the set of indices in $A$.
The number of roots (or trees, or blocks in $\Pi_q$) is distributed as the sum of $n$ Bernoulli random variables with success probabilities $\frac{q}{q+\lambda_i}$, for $i\leq n$, with the $\lambda_i$'s being the eigenvalues of $-L$, or their real parts, see e.g.~\cite[Prop. 2.1]{avena2018applications}. Its mean number is monotonically increasing in $q$. Further, these roots turn out to be well-distributed in the given network~\cite[Thm.1]{avena2018applications} and, conditional on the induced partition, their joint law is determined by the stationary measures of the random walk $X$ restricted to each block of the underlying partition~\cite[Prop.2.3]{avena2018applications}. 
These and other features of the LEP have been recently exploited to build novel algorithms for the following different applications in data science: multiresolution scheme, wavelets basis and filters for signal processing on graphs~\cite{avena2020wavelet,pilavci2020tikhonov,pilavci2020smoothing}, estimate traces of discrete Laplacians and other diagonally dominant matrices~\cite{barthelme2019trace}, 
network renormalization~\cite{avena2021solutions,avena2018randomforests}, centrality measures~\cite{chebotarev1997mft} and statistical learning~\cite{avrachenkov2017learning}.
These applications give further motivations to explore this LEP in more detail. 
Let us also stress that on certain geometrical setups such as the integers, it would be of interest to study the LEP in connection to other random partitions and a natural line of investigation would be to study its intruguing dynamical structure~\cite[see Thm.2 and Sect~2.2]{avena2018applications} within the theory of coalescence-fragmentation processes. 

\subsubsection*{\bf Other forest measures:}
To conclude this introduction let us clarify that this \emph{rooted} forest $\Phi_q$ should not be confused with other forest measures that have been receiving a large amount of attention in the literature in relation to universality classes in statistical physics and to negatively correlated systems. In particular, when taking the (weak) infinite volume limit of the UST on $d$-dimensional lattices for $d> 4$ (and other transitive settings), depending on the boundary condition procedure when approaching the limit, the resulting measure concentrates on \emph{unrooted} forests referred to as wired or free spanning forests, see e.g.~\cite{pemantle1991ust,benjamini2001usf, benjamini2004geometry,hutchcroft2017indistinguishability,hutchcroft2018interlacements} and references therein.
On finite graphs another natural extension of the UST is obtained when considering the uniform measure on unrooted spanning forests. Properties of this other fascinating forest measure have been recently investigated in ~\cite{bedini2009hyperforest,bauerschmidt2021hyperbolic}.

\subsubsection*{\bf Results overview and paper structure:} 
The rest of the paper is organized as follows. The statements of our main results are organized in Section~\ref{Results}.  
We start in Subsection~\ref{CorRusso} by stating a general characterization of monotone events in $q$, Theorem~\ref{thm:derivative_probabilities}. Therein, we also introduce the 2-point correlation function (which will later be analyzed in different graph settings to study the emergent partition) and assert in Theorem~\ref{thm:monotone_correlations_symmetric} its monotonicity on undirected graphs. We then explore in details ths LEP measure, by specializing on certain classes of graphs. 

In Subsection~\ref{GenResults} we look at general weighted directed trees. For this class we further extend the monotonicity result from Theorem~\ref{thm:monotone_correlations_symmetric}, see Theorem~\ref{thm:monotone_correlations_trees}, and we present an inclusion-exclusion reduction formula on arbitrary finite trees, see Proposition~\ref{prop:inex}. In Subsection~\ref{lines} we focus on the LEP on the first $n$ integers where equipartitions are favored. Formulas for the partition function are first derived in Thm~\ref{thm:path_partition}, and extended to a ring, Corollary~\ref{prop:cycle_path_partition}. Theorem~\ref{pathLEP} gives a recursive representation of the pairwise correlation in terms of reduced partition functions and offers bounds in terms of the correspondent RW on the infinite line. The subsequent Corollary~\ref{thm:path_correct_scaling} shows explicit bulk and boundary asymptotics.  
Section~\ref{trees} is then devoted to the exploration of the emergent blocks and detection of simple modular structures in tree-like structures by tuning the scale parameter $q$. In particular, Proposition~\ref{Star} and Theorem~\ref{CommStar} look at a star graph without and with a community structure, respectively. Proposition~\ref{thm:correlation_asymptotics_bounded_vertices} and Theorem~\ref{thm:correlation_asymptotics_regular_hierarchical} show similar analysis on finite trees with different weighted structures in which for different magnitudes of $q$ different layers are detected. Finally, in Theorem~\ref{prop:tcg_partition} we consider asymptotic detection in a bottleneck graph with two variable-size connected complete subgraphs by combining the results on the segment, after suitable contraction, and those for the mean-field case obtained in~\cite{avena2019meanfield}. All proofs are organized in the remaining sections. 
\section{Main results: monotonicities \& emergent partition on sparse graphs}\label{Results}

\subsection{Monotonicity \& two-point LEP potential}\label{CorRusso}
A notoriously difficult issue for most of the measures that can be obtained from 
FK-percolation, is to establish monotonicity properties as a function of the involved parameters. 
Our first theorem, which is reminiscent of Russo's pivotality formula in percolation models~\cite{grimmett1999percolation,debernardini2015russo}, offers a general characterization of monotone events w.r.t. $\Phi_q$ as a function of $q$.

\begin{theorem}[\bf{Monotone events for the rooted forest on arbitrary networks}]\label{thm:derivative_probabilities}
 Let $G=(V,E,w)$ be a weighted directed graph, and let 
 $r_q:=r(\Phi_q)$ be the number of roots of the random rooted forest $\Phi_q$.
 Then, for any set of rooted forests $\mathcal{H} \subseteq \mathcal{F}$, it holds that the derivative w.r.t $q$ of the probability of the event $\Phi_q \in \mathcal{H}$ is given by
 \begin{equation}
  \frac{d}{dq}\P(\Phi_q \in \mathcal{H})=\tfrac{1}{q}\P(\Phi_q \in \mathcal{H})
  \big[\E[r_q \mid \Phi_q \in \mathcal{H}] - \E[r_q]\big].
 \end{equation}
\end{theorem}
This statement is proven in Section~\ref{proofsGeneral} and shows that monotone events in $q$ are those for which the difference $\E[r_q \mid \Phi_q \in \mathcal{H}] - \E[r_q]$ has a constant sign as $q$ varies. In practice it might be not straightforward to check the sign of this difference, since it requires control on the conditional distribution of $r_q$. Still, for specific events we believe this statement can be of great help, of which we give an example in the proof of Theorem~\ref{thm:monotone_correlations_symmetric}. 
We also mention that in \cite{avena2018applications} a coupled version of the forest\footnote{This coupling corresponds to an explicit Markovian coalescence-fragmentation process with values in $\mathcal{F}$ in which coalescence of trees is dominant but whenever the underlying building RW produces a loop, a tree gets fragmented into subtrees, see~\cite[see Thm.2  and Sect2.2]{avena2018applications}.} 
is constructed by means of an algorithm allowing to sample an entire forest trajectory $(\Phi_q)_{q\in[0,\infty)}$. Yet, this coupling is monotone only in mean, but not trajectory-wise, hence this coupling is not useful to characterize monotone events. 

As anticipated, our main interest within this work is to explore monotonicity properties of this loop-erased partitioning and its detailed structure on trees and nearly-one-dimensional geometries. To do so, we will mainly analyze 2-point correlations associated to $\Pi_q$, which we introduce next. 
For a pair of distinct vertices $x,y\in V$, consider the event that these vertices belong to different blocks in $\Pi_q$. That is, the event $$\{B_q(x)\neq B_q(y) \}:=\{x \text{ and } y \text{ are in different blocks of } \Pi_q\},$$ where $B_q(z)$ stands for the block in $\Pi_q$ containing $z\in V$. 
\begin{definition}[{\bf 2-point correlations or pairwise LEP-interaction potential}]\label{FIP}
For given $q>0$ and $G$, and any pair $x,y\in V$, we call \emph{pairwise LEP-interaction potential} the following probability: 
\begin{align}\notag
 U_q(x,y):=&\P(B_q(x)\neq B_q(y))\\ 
 &=\sum_{\gamma}\P^{LE_q}_x(\gamma)\P_y(\tau_\gamma>\tau_q)\label{LEdec}
\end{align}
where $\tau_q$ denotes an independent exponential random variable of rate $q$, 
$\P_z$ and $\P_z^{LE_q}$ stand for the laws of the RW $X$ and the corresponding loop-erased RW killed at rate $q$, respectively, starting from $z\in V$. Further, the above sum runs over all possible self-avoiding paths $\gamma$ starting at $x$ and $\tau_\gamma:=\inf\{t\geq0: X_t \cap \gamma \neq \varnothing\}$ 
is the random walk hitting time of the set of vertices in $\gamma$. 
\end{definition}

The representation in Eq.~\eqref{LEdec} is a consequence of Wilson's sampling procedure and it holds true since, remarkably, this algorithm is exchangeable with respect to the starting point of each loop-erased random walk launched along the algorithm steps~\cite{wilson1996generating}. Furthermore, we notice that, as for any generic random partition of $V$, such an interaction potential defines a distance on the vertex set. This specific metric $U_q(x,y)$ can be interpreted as an affinity measure capturing how densely connected vertices $x$ and $y$ are in the graph $G$.

Our second general result, \cref{thm:monotone_correlations_symmetric}, further explores monotonicities in $q$ when considering undirected networks. Since spanning rooted forests impose a directionality on its edges, it is convenient to interpret an undirected graph as a symmetric directed graph with a symmetric weight function, $w(x,y)=w(y,x)$ for  $(x,y) \in E$. For these symmetric graphs  \cref{thm:monotone_correlations_symmetric} states that the ``unoriented'' edge process, see~\eqref{MonEdgeSet}, as well as the LEP-interaction potential, see~\eqref{MonPotential}, are both monontone in $q$. 
To state the result about the edge process, we will use the following notation. For a directed edge $e=(x,y)$ write $e^{-}=(y,x)$ to denote its reversed edge, and let $\cbrac{\pm A \subseteq \Phi_q}=\bigcap_{e \in A}(\cbrac{e \in \Phi_q}\cup \cbrac{e^{-} \in \Phi_q})$ denote
 the event that for each edge $e \in A$ either $e$ or $e^{-}$ is present in the random rooted forest $\Phi_q$.

\begin{theorem}[\bf{Monotonicity of edges and 2-point correlations on undirected networks}] \label{thm:monotone_correlations_symmetric}
Consider a symmetric weighted directed graph $G = (V,E,w)$ and the rooted forest $\Phi_q$ on $G$ for $q\in[0,\infty)$. Let $A \subseteq E$ be a set of directed edges, then the function
  \begin{equation}\label{MonEdgeSet}
    q \mapsto \P(\pm A  \subseteq \Phi_q)
  \end{equation}
  is monotone non-increasing.
Furthermore, for any distinct $x, y \in V$, the function
	\begin{equation}\label{MonPotential} q \mapsto	U_q(x,y)
	\end{equation}
is continuous and non-decreasing
with $U_0(x,y)=0$ and $\lim_{q\to\infty}U_q(x,y)=1$.
\end{theorem}

\begin{remark}[{\bf Main open problem}]
That this potential is in fact monotone, as expressed in \eqref{MonPotential}, is rather subtle. For example in \cite{avena2019meanfield} this fact was only checked for specific geometries via lengthy computations, while this general statement settles it immediately. Our proof of \cref{thm:monotone_correlations_symmetric} will exploit the undirectedness assumption, but we believe such monotonicity to be valid in great generality, though this remains a delicate open problem. In \cref{thm:monotone_correlations_trees} it is shown that \eqref{MonPotential} also holds for arbitrary weighted directed trees.
On the other hand, while \cref{MonPotential} might very well hold for all weighted directed graphs, it is not difficult to find examples of non-symmetric graphs for which the monotonicity of the (unoriented) edge process in \eqref{MonEdgeSet} fails. As an example consider the unweighted directed graph on four vertices with directed edge set $E=\{(1,2),(1,4),(2,3),(3,4)\}$. Then it holds that
\begin{equation*}
\P(\pm \cbrac{(1,2)} \subseteq \Phi_q) = \P((1,2) \in \Phi_q) = \tfrac{q^3+2q^2}{q^4+4q^3+5q^2+q},
\end{equation*}
which is increasing for $q<\sqrt{3}-1$.
\end{remark}

\subsection{Two-point-correlation on trees}\label{GenResults}
We start here to discuss results specific to trees. Let us notice that in this setup, the analysis is facilitated by the absence of cycles. In general, the mapping from $\mathcal{F}$ to rooted partitions is not injective, while on trees this is the case. So, on trees a rooted forest induces a unique rooted partition. For example in the constant weight case $w\equiv 1$, for a partition into $m\leq |V|$ blocks $\pi_m=\{B_1,B_2,\ldots,B_m\}\in \mathcal P(V)$, the measure in \cref{LEPmeas} reads as 
$$\mathbb{P}(\Pi_q=\pi_m)=  \frac{q^m \prod_{i=1}^m |B_i|}{Z(q)}, $$
from which we see that, for a given $q$, it concentrates on partitions where the block sizes tend to be of the same order. In this sense equipartitions are favored.

The first result in this tree specific setting extends the monotonicity of the LEP potential, as expressed in \cref{thm:monotone_correlations_symmetric}, to a specific weighted directed setting. 
As will become clear in \cref{ssec:monotonicity_proofs_symmetric,ssec:monotonicity_proofs_trees}, the proof is different than that of \cref{thm:monotone_correlations_symmetric}, as it relies on the absence of cycles.
\begin{theorem}[\bf{Monotonicity of 2-point correlations on trees}] \label{thm:monotone_correlations_trees}
If $G = (V,E,w)$ is a weighted directed tree, then for all $x,y \in V$ the function 
	\begin{equation*} 
	q \mapsto	U_q(x,y)
	\end{equation*}
is monotone non-decreasing.
\end{theorem}

Next we derive a representation of the LEP potential on arbitrary trees, in terms of reduced partition functions over subtrees.

To avoid confusion, in each statement in the sequel we will add proper indices to the partition functions and LEP-potential specifying the considered graph. The distance $d(x,y)$ between two vertices $x$ and $y$ will refer to the unweighted shortest path distance, i.e. the minimum number of edges on an undirected path between the two vertices.

\begin{proposition}[\bf{Inclusion-exclusion for 2-point-correlation on trees}]\label{prop:inex}
Let $G = (V,E,w)$ be a weighted directed tree.
Fix $x, y \in V$ with $d(x,y)=d$ and let $(z_i)_{i=0}^d$ be the unique undirected path 
with $z_0 = x$ and $z_d = y$. 
For a subset $I \subseteq [d]$ let $G_{I}$ denote the graph obtained by removing all edges 
between $z_{i-1}$ and $z_i$ from $G$ for all $i \in I$.
Denote the $\abs{I}+1$ connected components of $G_{I}$ by $G^{1}_{I},\ldots,G^{\abs{I}+1}_{I}$.
Then, for every $q>0$, the following representation is valid
\begin{equation}
 U_q^{(G)}(x,y)=\frac{1}{Z_G(q)} \paren*{\sum_{k=1}^{d}(-1)^{k+1} \sum_{I \in \binom{[d]}{k}} \prod_{i=1}^{k+1} Z_{G_{I}^{i}}(q)}.
\end{equation}
Here $\binom{[d]}{k}$ denotes the collection of $k$-element subsets of $[d]$.

In particular for $x,y$ such that $d(x,y) = 1$:
\begin{equation}\label{eq:close_factor}
  U_q(x,y) = \frac{Z_x(q)Z_y(q)}{Z_G(q)},
\end{equation}
where $Z_x(q)$ and $Z_y(q)$ denote the partition functions of 
the two connected components of the graph obtained by removing the edges between $x$ and $y$.
\end{proposition}

\subsection{Integer partitioning: analysis on lines and rings}\label{lines}
In what follows we denote by $PG_n:=\Z\cap[1,n]$ the (undirected and unweighted) path-graph constituted by the first $n$ integers and by $CG_n$ the cycle-graph on $n$ vertices (i.e. the one dimensional discrete torus).

\begin{theorem}[\bf{Partition function of path-graphs}]\label{thm:path_partition}
The partition function in~\eqref{Z} of $PG_n$ can be expressed in the following ways:
 \begin{align}
 Z_{PG_n}(q) &= \sum_{k=1}^{n}\binom{n+k-1}{2k-1}q^{k} \label{eq:path_partition_combinatorial} \\
 &= \prod_{k=1}^{n}\paren*{q+2-2\cos\paren*{\tfrac{\pi(n-k)}{n}}} \label{eq:path_partition_spectral} \\
  &= \frac{q\paren*{q+2+\sqrt{q^2+4q}}^{n}-q\paren*{q+2-\sqrt{q^2+4q}}^{n}}{2^{n}\sqrt{q^2+4q}} \label{eq:path_partition_recurrence}\\
  &= qU_{n-1}(\tfrac{q}{2}+1) \label{eq:path_partition_chebyshev}.
\end{align}
Here $U_{n-1}$ denotes the $n-1$-th degree Chebyshev polynomial of the second kind.
\end{theorem} 

As can be appreciated in the proof, the above different representations reflect different computational methods suited for the random forest.
We notice that for $q=1$ evaluating this partition function corresponds to counting the number of rooted forests of the path-graph, as previously derived in~\cite{chebotarev2008goldenratio}.

One of the messages of this paper is that having an explicit characterization of a simple given geometry can be useful to derive information on some more involved geometry. The next corollary shows one such very simple instance by expressing the partition function on the torus in terms of partition functions of the simpler path-graph.

\begin{corollary}[\bf {Partition function of cycle-graphs}]\label{prop:cycle_path_partition}
 The partition function of $CG_n$ is given by
 \begin{align}
  Z_{CG_n}(q) &= Z_{PG_n}(q)+\tfrac{2}{q}\left[Z_{PG_n}(q)-Z_{PG_{n-1}}(q)\right]-2 \label{eq:cycle_partition_path} \\
  &=\sum_{k=1}^{n}\paren*{\binom{n+k}{2k}+\binom{n+k-1}{2k}}q^k. \label{eq:cycle_partition_combinatorial}
 \end{align}
\end{corollary}

\begin{theorem}[\bf{Correlations on path-graph and bounds via random walk on $\Z$}]\label{pathLEP}
Let $x,y \in [n]$ be two vertices in $PG_n$ at distance $d:=y-x>0$.
Then, for any $q>0$, the 2-point correlation between $x$ and $y$ is given by
\begin{equation}\label{eq:path_correlation}
\correlation[PG_n]{q}{x}{y}=1-\frac{Z_{PG_{n-d}}(q)}{Z_{PG_{n}}(q)}-\frac{d \left[Z_{PG_x}(q)-Z_{PG_{x-1}}(q)\right]
\left[Z_{PG_{n-y+1}}(q)-Z_{PG_{n-y}}(q)\right]}{qZ_{PG_n}(q)}. 
\end{equation} 
Moreover, by denoting with $S=(S_m)_{m \in \N_0}$ the discrete-time simple random walk on $\Z$ starting at $0$, the following bounds are satisfied
\begin{equation} 
 \paren*{1-\paren*{\tfrac{2}{2+q}}^{m}}^2\paren*{2\P(|S_{m}| < \tfrac{d}{2})-1}^2 \leq
 \correlation[PG_n]{q}{x}{y} \leq 1- \P(|S_{m}| > d) \paren*{\tfrac{2}{2+q}}^{m},
\end{equation}
where the upper bound is valid for any $m\in\N$, while the lower bound holds for $m$ such that $\P(|S_{m}| < \tfrac{d}{2}) \geq \tfrac{1}{2}$.
\end{theorem}

From the above statement, due to the diffusive behavior of the simple random walk $S$, it is clear that the correlation function between two points in a segment is non-degenerate when $q_n$ scales with the inverse square distance between the two points. 
The next corollary makes this statement precise and shows that boundary effects emerge neatly from the asymptotic analysis.

\begin{corollary}[\bf{Non-degenerate scaling and asymptotic boundary effects}]\label{thm:path_correct_scaling}
For each $n \in \N$ let $x_n$ and $y_n$ be vertices in $PG_n$. 
Let $d_n$ denote the distance between these vertices and let $(q_n)_{n \in \N}$ be a monotone sequence of positive parameters.
Then, if the limit $\lim_{n\to\infty}U^{(PG_n)}_{q_n}(x_n, y_n)$ exists, it holds that
$$ \lim_{n\to\infty}U^{(PG_n)}_{q_n}(x_n, y_n)\in (0,1) 
 \text{            if and only if             }  
 q_n = \tfrac{c}{d_n^2}+\smallO(\tfrac{1}{d_n^2}) \text{ for some constant $c>0$.}$$

In particular, fix $\delta > 0$  and let $(\zeta_n)_{n \in \N}$ be a sequence such that 
 $\zeta_n \in [\delta\sqrt{n}, n-\delta\sqrt{n}]$ for large enough $n$.
Set $x_n=\zeta_n-\delta\sqrt{n}+\smallO(\sqrt{n})$, $y_n=\zeta_n+\delta\sqrt{n}+\smallO(\sqrt{n})$
and $q_n\sim \tfrac{1}{d_n^2}$, then the following two limits, distinguishing between the bulk and near the boundaries, are possible:

\begin{equation}\label{eq:path_asymptotics_root_distance}
 \lim_{n\to\infty}
 \correlation[PG_n]{q_n}{x_n}{y_n} =
 \begin{cases}
1-\frac{3}{2e}
 & \text{ if  } \ \zeta_n = \smallOmega(\sqrt{n}) \text{ and } \zeta_n = n-\smallOmega(\sqrt{n}) \\
1-\frac{3}{2e}-\frac{1}{2}e^{-\frac{\alpha}{\delta}}
 & 
\text{ if  } \ \zeta_n = \alpha \sqrt{n}+\smallO(\sqrt{n}) \text{  or  } \zeta_n = n -\alpha \sqrt{n}+\smallO(\sqrt{n})
\text{ for some } \alpha\geq\delta.
\end{cases}
\end{equation}
\end{corollary}

In the above statement we computed the exact asymptotics only when the distance of the two vertices scales as the square root of $n$. Similar exact computations can be derived for other choices of the magnitude of this distance. We refer the interested reader to 
\cite{koperberg2020lep} for analogous statements in the cases when $d_n$ stays of order one or diverges linearly.
In particular, we note that \emph{giants} (i.e. blocks of order $|V|$) appear at scale $q_n\sim d_n^{-2}$ and a unique giant emerges as soon as $q_n= o(n^{-2})$.

\subsection{Detecting modular structures in stylized tree-like geometry}\label{trees}
We collect here a number of simple statements of different flavour aiming to illustrate that in tree-like graphs the emergence of giants and other modular structures can be detected with high probability by tuning $q$. 
\Cref{fig:overview_detectability} gives a graphical overview of the main results in this section, which are given in 
\cref{CommStar,thm:correlation_asymptotics_regular_hierarchical,prop:tcg_partition}. We start in \cref{Star} by making precise how $q$ scales on a given large star graph.

\begin{figure}[h!ptb]
\centering
    \begin{subfigure}{0.3\textwidth}
        \includegraphics[width=\hsize]{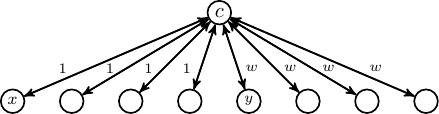}
        \caption{\scriptsize A community star graph with $n=9$ vertices, with center vertex $c$ and $k=4$ vertices in the weight-1 community. The remaining four vertices belong to the weight-$w$ community.}
        \label{fig:community_star}
    \end{subfigure}
\hfill
    \begin{subfigure}{0.3\textwidth}
    \includegraphics[width=\hsize]{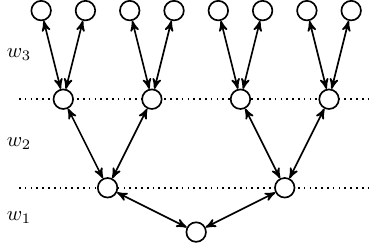}
    \caption{\scriptsize A $2$-regular tree of height $h=3$ with hierarchical edge weights.
 Each edge in generation $i$ has weight $w_i$ and these weights satisfy $w_1 \leq w_2 \leq w_3$.}
    \label{fig:regular_hierarchical_tree}
\end{subfigure}
\hfill
    \begin{subfigure}{0.3\textwidth}
    \includegraphics[width=\hsize]{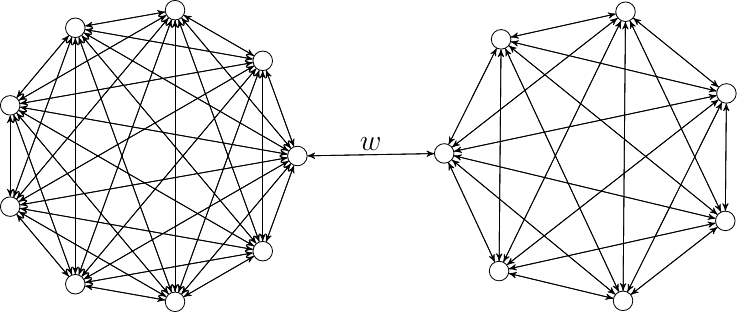}
    \caption{\scriptsize A bottleneck graph with bridge weight $w$ and two cliques of size $n = 9$ and $m = 7$.}
    \label{fig:bottleneck_graph}
\end{subfigure} \\
    \begin{subfigure}{0.3\textwidth}
        \includegraphics[width=\hsize]{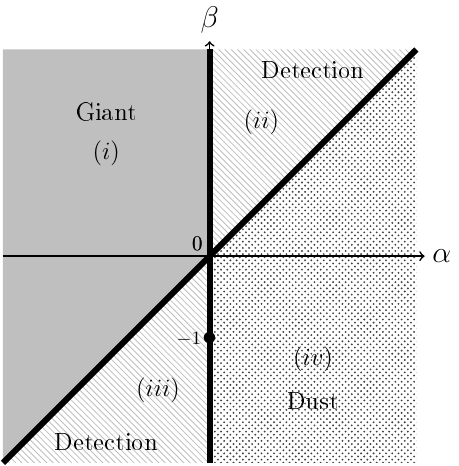}
        \caption[]{\scriptsize Phase diagram for the community star graph, with $q=n^{\alpha}$ and $w=n^{\beta}$. 
  For each of the regions the following event occurs with high probability:
  $(i)$ One single tree; 
      $(ii)$  $k+1$ trees, all $k$ vertices incident to a weight $1$ edge are isolated, while the remaining vertices form a single tree; 
      $(iii)$ $n-k$ trees, all $n-k-1$ vertices incident to a weight $w$ edge are isolated, while the remaining vertices form a single tree; 
      $(iv)$ $n$ isolated vertices. 
   The exact limit values of the correlations along the bold lines, i.e. in the non-degenerate regimes, 
   can be found in in \cref{CommStar}. 
    }
        \label{fig:CommStardetect}
    \end{subfigure}
\hfill
    \begin{subfigure}{0.3\textwidth}
    \includegraphics[width=\hsize]{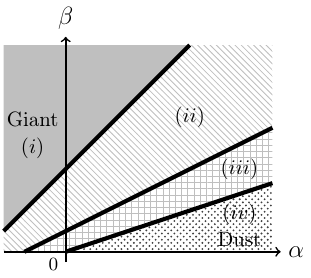}
    \caption{\scriptsize Phase diagram for the $d$-regular hierarchical tree of height $h=3$, with $d=n$, 
    $q=n^{\alpha}$ and $j$-th generation edge weights $w_j=n^{j\beta}$ for $\beta \geq 0$. 
    For each of the regions the following event occurs with high probability:
      $(i)$ One single tree; 
      $(ii)$  All 2nd and 3rd generation edges are present, while all 1st generation edges are absent; 
      $(iii)$ All 3rd generation edges are present, while all 1st and 2nd generation edges are absent;
      $(iv)$ All $1+n+n^2+n^3$ vertices are isolated. 
    }
    \label{fig:hierarchical_tree_detectability}
\end{subfigure}
\hfill
    \begin{subfigure}{0.3\textwidth}
    \includegraphics[width=\hsize]{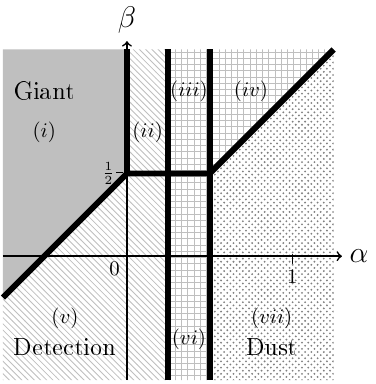}
  \caption{\scriptsize Phase diagram for the bottleneck graph, with $q=n^{\alpha}$, $w=n^{\beta}$ and $m=\sqrt{n}$. 
    Regions $(ii)$ and $(v)$ are the regimes where the LEP detects the community structure.
    For each of the regions the following event occurs with high probability:
      $(i)$ One single tree;
      $(ii)$ Two trees on $n+1$ and $\sqrt{n}$ vertices, with the large tree containing both bridge vertices;
      $(iii)$ One tree consists of the $n$ vertices in the largest clique with the bridge vertex from the small clique, while the other vertices in the small clique are isolated;
      $(iv)$ Both bridge vertices are connected, and all others are isolated; 
      $(v)$ Two trees with $n$ and $\sqrt{n}$ vertices, while the bridge edge is absent;
     $(vi)$ One tree with all $n$ vertices in the largest clique, while the $\sqrt{n}$ vertices in the small clique are isolated;
    $(vii)$ $n+\sqrt{n}$ isolated vertices.
  }
    \label{fig:tcg_detectability}
\end{subfigure} 
    \caption{An overview of the results in \cref{trees}. The \cref{fig:community_star,,fig:regular_hierarchical_tree,,fig:bottleneck_graph} depict the geometries treated in \cref{CommStar,,thm:correlation_asymptotics_regular_hierarchical,,prop:tcg_partition}, respectively, while 
    \cref{fig:CommStardetect,,fig:hierarchical_tree_detectability,,fig:tcg_detectability} give a graphical representation of their results.
    }
    \label{fig:overview_detectability}
\end{figure}

\begin{proposition}[\bf{Potential and its limit on a homogeneous star graph}]\label{Star}
Let $SG_n$ denote the star graph on $n$ vertices, i.e. 
$SG_n$ is an undirected tree consisting of a single center vertex $c$ that is adjacent to $n-1$ leaves.
Let $x,y$ be two distinct leaves and equip $SG_n$ with a uniform weight function that assigns weight $w$ to all edges.
Given $q > 0$, 
\begin{align}
  U_q(c,x) &= \frac{q(q + (n-1)w)}{(q+w)(q + nw)} \label{eq:star_correlation_center}\\
  \nonumber \\
  U_q(x,y) &= \frac{q(q^2 + (n+2)w q + 2(n-1)w^2)}{(q+w)^2(q + nw)}, \label{eq:star_correlation_leaves}
\end{align}
which implies that $q \mapsto U_q(c,x)$ and $q \mapsto U_q(x,y)$ are strictly concave. 

Let $q_n = \bar{q}n^{\alpha}$ and $w_n = \bar{w}n^{\beta}$ with $\alpha, \beta \in \R$ 
and $\bar{q},\bar{w} \in (0, \infty)$. Then
\begin{equation}\label{As1}
  \lim_{n\to\infty} U_{q_n}^{(SG_n)}(c,x) = 
  \left\{
  \begin{array}{ll}
    1 & \alpha > \beta \\
    \frac{\bar{q}}{\bar{q} + \bar{w}} & \alpha = \beta \\
    0 & \alpha < \beta
  \end{array}\right.
\end{equation}
and 
\begin{equation}\label{As2}
  \lim_{n\to\infty} U_{q_n}^{(SG_n)}(x,y) =
  \left\{
  \begin{array}{ll}
    1 & \alpha > \beta \\
    \frac{\bar{q}(\bar{q} + 2\bar{w}^2)}{(\bar{q} + \bar{w})^2} & \alpha = \beta \\
    0 & \alpha < \beta
  \end{array}\right.
\end{equation}
\end{proposition}

We see that the the critical phase for the appearance of a giant is when $\alpha=\beta$ for which the resulting connected subtree can be thought of as a star whose center has offspring distribution of parameter $\bar{q}/(\bar{q} + \bar{w})$, while a unique giant emerges as soon as $\alpha<\beta$. 

The following statement clarifies how $q$ should be scaled in a non-homogeneous star to detect an implanted sub-module of leaves more densely connected to the center. \Cref{fig:CommStardetect} offers a graphical representation of Theorem~\ref{CommStar}.

\begin{theorem}[\bf{Asymptotic detection in a star graph with two communities}]\label{CommStar}
Let $CSG_{n,k}$ denote the community star graph on $n$ vertices, which is a star graph on $n$ vertices equipped with an inhomogeneous weight function, that assigns weight $1$ to $k$ edges and weight $w$ to the remaining $n-k-1$ edges, as depicted in \cref{fig:community_star}.
Let $c$ denote the center vertex, $x,y$ vertices incident to an edge with weight $1$ and $w$, respectively.
For $\alpha,\beta \in \R$ take $q_n = n^{\alpha}$, $w_n = n^{\beta}$ and $k$ constant. Then
\begin{equation}
  \lim_{n \to \infty} U_{q_n}^{(CSG_{n,k})}(c,x) = 
  \left\{
  \begin{array}{ll}
    0 & \alpha < 0 \\
    & \\
    \left\{
    \begin{array}{ll}
      \frac{1}{2} & \beta > -1 \\
      & \\
      \frac{k+3}{2k + 8} & \beta = -1 \\
      & \\
      \frac{k+1}{2k+4} & \beta < -1
    \end{array}
    \right. & \alpha = 0 \\
    & \\
    1 & \alpha > 0
  \end{array}
  \right.
\end{equation}
and
\begin{equation}	
  \lim_{n \to \infty} U_{q_n}^{(CSG_{n,k})}(c,y) = 
  \left\{
  \begin{array}{ll}
    0 & \alpha < \beta \\
    \frac{1}{2} & \alpha = \beta \\
    1 & \alpha > \beta
  \end{array}
  \right.
\end{equation}
\end{theorem}

The next two statements show similar detections on trees of different flavours.
\begin{proposition}[\bf{Asymptotic correlation in undirected trees with a bounded number of vertices}]\label{thm:correlation_asymptotics_bounded_vertices}
Let $G=(V,E)$ be an undirected tree and let $w_k:E\to(0,\infty)$ be a sequence of edge weight functions. Write $G_k=(V,E,w_k)$ to denote the weighted graph obtained by equiping $G$ with $w_k$. For each $k \in \N$ let $q_k>0$ be an intensity parameter and assume that for each edge 
$e\in E$ the limit $\lim_{k \to \infty}\tfrac{w_k(e)}{q_k}$ exists in $[0, \infty]$.
Let $x,y \in V$ be two adjacent vertices.
Then, as $k \to \infty$ it holds that
\begin{equation}
 \correlation[G_k]{q_k}{x}{y} \to 
 \begin{cases}
  0 & \text{ if }q_k = \smallO(w_k(x,y)) \\
  1 & \text{ if }q_k = \smallOmega(w_k(x,y)).
 \end{cases}
\end{equation}
\end{proposition}

The following theorem holds for a specific class of undirected weighted trees that will be called `hierarchical trees'.
In these trees one vertex is specified as \emph{ancestor} vertex.
The \emph{height} or \emph{generation} of a vertex or edge is its distance to the ancestor. 
A \emph{hierarchical tree} is a tree with edge weights $w:E \to [0,\infty)$ satisfying the following two properties:
\begin{enumerate}[$(i)$]
 \item if $e,e' \in E$ are edges in the same generation of the regular tree, then $w(e)=w(e')$; 
 \item if $e_i,e_j \in E$ are edges in generations $i$ and $j$ with $i<j$, respectively, then $w(e_i)\leq w(e_j)$. 
\end{enumerate}
So, edges further from the ancestor of the hierarchical tree have more weight.

The \emph{height of the tree} is the maximal height of its vertices. 
If $x$ is a vertex at height $h$ and $y$ is a neighbor of $x$ at height $k-1$, 
then we call $x$ a \emph{child} of $y$ and $y$ the \emph{parent} of $x$.
If each vertex with height less than the height of the tree has $d$-children, then we call the tree $d$-\emph{regular}.
The \emph{ancestry} of a vertex is the unique path from the vertex to the ancestor (including the vertex itself).
A depiction of a regular hierarchical tree is given in \cref{fig:regular_hierarchical_tree}.

\begin{theorem}[\bf{Asymptotic detection of layers in a regular hierarchical weighted tree}] \label{thm:correlation_asymptotics_regular_hierarchical}
For each $n \in \N$ let $G_n=(V_n,E_n,w_n)$ be a undirected $d_n$-regular tree with hierarchical edge weights.
For each $n \in \N$ let $x_n,y_n \in V_n$ be vertices such that $x_n$ is the parent of $y_n$ 
and such that the minimal distance between $y_n$ and a leaf of $G_n$ is constant in $n$.
Denote this constant distance by $k$. Let $e_n$ denote the edge between $x_n$ and $y_n$.
For each $n \in \N$ let $q_n>0$ be the intensity parameter.
Then as $n \to \infty$ it holds for the 2-point correlation between $x_n$ and $y_n$ that
\begin{equation}
 \correlation[G_n]{q_n}{x_n}{y_n} \to 
 \begin{cases}
  0 & \text{ if } q_n = \smallO\left(d_n^{-k}w_{n}(e_n)\right) \\
   1 & \text{ if } q_n = \smallOmega\left(d_n^{-k}w_{n}(e_n)\right). 
 \end{cases}
\end{equation}
\end{theorem}

We conclude by showing with an illustrative example how the analysis on trees presented here and those on complete graphs pursued  in~\cite{avena2019meanfield} can be combined to obtain results on mixed geometrical setups. 
The resulting regimes are summarized in the phase diagram in \Cref{fig:tcg_detectability}.

\begin{theorem}[{\bf Detection of cliques in a bottleneck graph}]\label{prop:tcg_partition}
 Let $BG_{n,m}$ be a bottleneck (two-cluster) graph. That is, an undirected graph consisting of two disjoint cliques $C_1, C_2$ on $n$ and $m$ vertices, respectively, that are connected via a single bridge edge, as depicted in \cref{fig:bottleneck_graph}.
 Equip $BG_{n,m}$ with a weight function that assigns weight $w$ to the bridge and weight $1$ to all other edges.
 Then its partition function is given by
 \begin{equation}
  Z(q)=q\paren*{q(q+n)(q+m) + w(q+1)(2q+n+m)}(q+n)^{n-2}(q+m)^{m-2}. \label{eq:tcg_partition} 
 \end{equation}

Further, set $q=q_n > 0$ and let $w=w_n$ and $m=m_n$ depend on $n$ where $n \geq m$. 
Denote by $b,b'$ the two vertices incident to the bridge, by $x,x'$ two vertices that both belong to the clique $C_i$ containing $b$, 
and by $y$ a vertex in the clique containing $b'$.
Then as $n \to \infty$ it holds for the 2-point correlation between these vertices that
 \begin{align}
  U_q(x,x') &\to 
  \begin{cases}
   0 & \text{ if }q=\smallO(\sqrt{\abs{C_i}}) \\
   1 & \text{ if }q=\smallOmega(\sqrt{\abs{C_i}})
  \end{cases} \label{eq:tcg_cluster_correlation} \\ 
   U_q(b,b') &\to 
  \begin{cases}
   0 & \text{ if }q=\smallO(\tfrac{w}{m}) \text{ or }(q=\smallO(w), \ w=\smallOmega(m))\\
   1 & \text{ if }q=\smallOmega(w)\text{ or }(q=\smallOmega(\tfrac{w}{m}), \ w=\smallO(m))
  \end{cases} \label{eq:tcg_bridge_correlation} \\ 
  U_q(b,x)&\to 
  \begin{cases}
   0 & \text{ if }q=\smallO(1) \text{ or }(q=\smallO(\sqrt{\abs{C_i}}), \ w=\smallO(m)) \text{ or }(q=\smallO(\sqrt{\abs{C_i}}), \ m=\smallO(n)) \\
   \tfrac{c}{1+c} & \text{ if }q=\smallOmega(1), \ q=\smallO(\sqrt{\abs{C_i}}), \ w=\smallOmega(m), \ \abs{C_i}=n, \ m\sim cn \text{ with }c\in (0,1] \\
   \tfrac{1}{1+c} & \text{ if }q=\smallOmega(1), \ q=\smallO(\sqrt{\abs{C_i}}), \ w=\smallOmega(m), \ \abs{C_i}=m, 
   \ m\sim cn \text{ with }c\in (0,1] \\
   1 & \text{ if }q=\smallOmega(\sqrt{\abs{C_i}})
  \end{cases} \label{eq:tcg_bridge_cluster_correlation} \\
  U_q(x,y)&\to 
  \begin{cases}
   0 & \text{ if }q=\smallO(1), \ q= \smallO(\tfrac{w}{m}) \\
   1 & \text{ if }q=\smallOmega(1) \text{ or }(q=\smallO(1), \ q = \smallOmega(\tfrac{w}{m})).
  \end{cases} \label{eq:tcg_different_cluster_correlation} 
 \end{align}
\end{theorem}

\section{Proofs of results on general graphs}\label{proofsGeneral}
\subsection{Monotone events in terms of number of roots}
\begin{proof}[Proof of Theorem~\ref{thm:derivative_probabilities}]
Let $L$ be the graph Laplacian of $G$.
 Write $n = \abs{V}$ and let $\lambda_1,\ldots,\lambda_n$ denote the eigenvalues $-L$.
 By \cite[proposition 2.1]{avena2018applications} it holds that
  \begin{equation}
   \E[r_q] = \sum_{i=1}^{n}\frac{q}{q+\lambda_i} = \frac{q Z'(q)}{Z(q)},
  \end{equation}
so that the derivative of the partition function is given by
\begin{equation}
 Z'(q)=\tfrac{1}{q}\E[r_q]Z(q).
\end{equation}
Note that the conditional probability $\P(\Phi_q \in \mathcal{H} \mid r_q = k)$ does not depend on $q$.
Also, the probability $\P(r_q=k)$ can be written as $\tfrac{c_k q^k}{Z(q)}$, where $c_k$ is some constant independent of $q$, corresponding to the coefficent of degree $k$ of the characteristic polynomial in~\eqref{Z}.
Hence, we have that
\begin{align*}
 \frac{d}{dq}\P(\Phi_q \in \mathcal{H}) &= \frac{d}{dq} \sum_{k=1}^{n} \P(\Phi_q \in \mathcal{H} \mid r_q=k)\P(r_q=k) \\
 &= \sum_{k=1}^{n} \P(\Phi_q \in \mathcal{H} \mid r_q=k) c_k \frac{d}{dq} \frac{q^k}{Z(q)} = \sum_{k=1}^{n} \P(\Phi_q \in \mathcal{H} \mid r_q=k) c_k \frac{kZ(q)q^{k-1}-q^kZ'(q)}{Z(q)^2} \\
 &= \tfrac{1}{q}\sum_{k=1}^{n} \P(\Phi_q \in \mathcal{H} \mid r_q=k)c_k \frac{kq^{k}-q^{k}\E[r_q]}{Z(q)} 
= \tfrac{1}{q}\sum_{k=1}^{n} \P(\Phi_q \in \mathcal{H} \mid r_q=k)\P(r_q=k) \paren*{k-\E[r_q]} \\
 &= \tfrac{1}{q}\P(\Phi_q \in \mathcal{H})\sum_{k=1}^{n} \P(r_q=k \mid \Phi_q \in \mathcal{H}) \paren*{k-\E[r_q]} = \tfrac{1}{q}\P(\Phi_q \in \mathcal{H})\paren*{\E[r_q \mid \Phi_q \in \mathcal{H}] - \E[r_q]},
\end{align*}
where in the last step we use that $\sum_{k=1}^{n} \P(r_q=k \mid \Phi_q \in \mathcal{H})=1$.
\end{proof}

\subsection{Some reduction/extension lemmas}
We introduce here some rather classical contraction tools. Though, we stress that the following definition of contraction is slightly different from what is often encountered in the UST literature, as it is adapted to the setting of weighted directed graphs.
\begin{definition}[Directed edge contraction]
 Let $G=(V,E,w)$ be a weighted directed graph and $e \in E$ a directed edge from vertex $x$ to $y$, i.e. $e=(x,y)$.
 The graph $G\vec{/}e$ obtained by performing \emph{the directed edge contraction} in $G$ over edge $e$ is the graph 
 obtained by first removing all outgoing edges of $x$ and then contracting $x$ and $y$ into a single vertex, 
 while retaining all outgoing edges from $y$ and all ingoing edges to both $x$ and $y$.
 
 If $B$ is a set of edges that constitutes a rooted forest of $G$, then the operations of performing a directed edge contraction on different edges in $B$ commute. Thus for such a $B$ we can define the graph $G\vec{/}B$ to be the graph obtained by performing directed edge contractions on all edges in $B$.
\end{definition}
Besides this notation for directed edge contractions, we will also use the standard notation $G-e$ to denote the graph obtained by removing the directed edge $e$ (without removing the reversed edge), and $G/e$ to denote a regular edge contraction over edge $e$, i.e. $G/e$ is the graph obtained by identifying the two endpoints of $e$ as a single vertex.

\begin{lemma}[\bf{Various expressions for edge probabilities}] \label{lem:edge_probability}
Let $G=(V,E,w)$ be a weighted directed graph and $e=(x,y)$ a directed edge from vertex $x$ to $y$. Let $R_q$ be the set of root vertices of $\Phi_q$.
Let $L$ denote the graph Laplacian of $G$ and $K_q$ the RW Green's kernel given by 
$K_q=q(qI-L)^{-1}$.
For each directed edge $e$ write $G\vec{/}e$ to denote the directed $e$-contraction of $G$. Then it holds that
\begin{equation*}
 \P(e \in \Phi_q) = \tfrac{w(e)}{q}\P(x \in R_q, \ x \nleftrightarrow_{\Phi_q} y) = 
 \tfrac{w(e)}{q}(K_q(x,x)-K_q(y,x))= w(e)\frac{Z_{G\vec{/}e}(q)}{Z_{G}(q)}. 
\end{equation*}
\end{lemma}
\begin{proof}
 Let $e=(x,y)$ be an edge from $x$ to $y$.
 Let $\mathcal{A}=\cbrac{F \in \mathcal{F}_G \colon e \in F}$ denote the set of rooted forests of $G$ that do contain edge $e$.
 Write $\mathcal{H}=\cbrac{F \in \mathcal{F}_G \colon x \in R(F), \ x \nleftrightarrow_{F} y}$ to denote the set of forests in which $x$ is a root that is not connected to $y$. Note that there is a one-to-one correspondence $f:\mathcal{A}\to\mathcal{H}$ given by $f(F)=F-e$.
Moreover, it holds that $w(F)=w(e)w(f(F))$ and that $r(F)=r(f(F))-1$, where $r(F)$ denotes the number of roots of $F$.
The first identity follows by summation over all forests in $\mathcal{A}$.
For the second identity we use the Chebotarev-Shamis matrix-forest theorem \cite{chebotarev1997mft}, which states that $K_q(y,x) = \P(x \in R_q, \ x \leftrightarrow_{\Phi_q} y)$.
The third identity follows by considering the bijection $g:\mathcal{H} \to \mathcal{F}_{G\vec{/}e}$ that sends all edges of a forest in $\mathcal{H}$ to their corresponding edges in $G\vec{/}e$. Note that here $G\vec{/}e$ could be a multigraph. This bijection satisfies $w(F)=w(g(F))$ and $r(F)=r(g(F))+1$, so that summation over all forests in $\mathcal{H}$ yields the result.
\end{proof}

The following lemma shows the well-known spatial Markov property for the UST, see e.g. \cite{hutchcroft2019planar}, tailored to the rooted forest measure $\Phi_q$. 

\begin{proposition}[\bf{Spatial Markov property}] \label{lem:spatial_markov_property}
Let $G=(V,E,w)$ be a weighted directed graph and $A,B \subseteq E$ two disjoint sets of directed edges. Then it holds for all $F \in \mathcal{F}_{G}$ with $F \cap A=\varnothing$ and $B \subseteq F$ that
\begin{equation}\label{eq:spatial_markov_property}
\P^{(G)}(\Phi_q=F \mid \Phi_q \cap A = \varnothing, \ B \subseteq \Phi_q) = \P^{((G-A)\vec{/}B)}(\Phi_q = F \vec{/} B).
\end{equation}
For any edge $e \in E$ the partition function of $G$ satisfies the deletion-contraction identity
\begin{equation}\label{eq:deletion_contraction}
 Z_G(q)=Z_{G-e}(q)+w(e)Z_{G\vec{/}e}(q).
\end{equation}
Moreover, if $G$ is a symmetric graph, then it holds that 
\begin{equation}\label{eq:spatial_markov_property_undirected}
\P^{(G)}( \Phi_q = F \mid \Phi_q \cap A = \varnothing,  \pm B \subseteq \Phi_q) = \P^{((G-A)/B)}(\Phi_q = F/B),
\end{equation}
where $G/B$ denotes the regular edge contraction of all edges in $B$.
\end{proposition}
\begin{proof}
It is sufficient to show that the statement holds when $\abs{A \cup B}=1$, since the general statement then follows by induction.
First assume that $B = \varnothing$ and $A=\cbrac{e}$ for some edge $e \in E$. 
Let $\mathcal{A}=\cbrac{F \in \mathcal{F}_G \colon e \notin F}$ denote the set of rooted forests of $G$ that do not contain edge $e$.
Write $r(F)$ to denotes the number of roots of the rooted forest $F$.
There is a natural one-to-one correspondence $f:\mathcal{A}\to\mathcal{F}_{G-e}$ given by $f(F)=F$.
Hence, we have for all $F \in \mathcal{A}$ that
\begin{align*}
 \P^{(G)}(\Phi_q=F \mid e \notin \Phi_q) 
 &= \frac{\P^{(G)}(\Phi_q=F)}{\P^{(G)}(e \notin \Phi_q)}= \frac{q^{r(F)}w(F)}{\sum_{H \in \mathcal{A}}q^{r(H)}w(H)}= \frac{q^{r(F)}w(F)}{\sum_{H \in \mathcal{F}_{G-e}}q^{r(f^{-1}(H))}w(f^{-1}(H))} \\
 &= \frac{q^{r(F)}w(F)}{\sum_{H \in \mathcal{F}_{G-e}}q^{r(H)}w(H)} = \frac{q^{r(F)}w(F)}{Z_{G-e}(q)}= \P^{(G-e)}(\Phi_q=F). 
\end{align*}

Assume instead that $A = \varnothing$ and $B=\cbrac{e}$ for some edge $e \in E$. 
Then by \cref{lem:edge_probability} we have for all $F\in \mathcal{F}_{G}$ with $e \in F$ that
\begin{equation*}
 \P^{(G)}(\Phi_q=F \mid e \in \Phi_q) = \frac{\P^{(G)}(\Phi_q=F)}{\P^{(G)}(e \in \Phi_q)} 
 = \frac{q^{r(F)}w(F)}{w(e)Z_{G\vec{/}e}(q)} 
 = \P^{(G\vec{/}e)}(\Phi_q=F/e). 
\end{equation*}

The proof of \cref{eq:deletion_contraction} is analogous to that of \cref{eq:spatial_markov_property}, while \cref{eq:spatial_markov_property_undirected} follows directly from the spatial Markov property for the UST.
\end{proof}

\Cref{lem:graph_extension_single_vertex,lem:graph_extension_single_edge} both represent the same simple combinatorial manipulation, but in two slightly different settings. The same manipulation can be extended beyond the simple setups of these lemmas, but for notational simplicity we stick to these versions, which are tailored to sparse geometries. 

These lemmas are phrased in terms of the \emph{non-normalized} rooted forest measure defined as 
\begin{equation}\label{eq:non-normalized_measure}
\NNM^{(G)}(\Phi_q \in \cdot)=Z_G(q)\P^{(G)}(\Phi_q \in \cdot).  
\end{equation}
This measure has the benefit that the measure of a rooted forest dependends on the geometry of the underlying graph only through the total number of vertices. That is, for any rooted forest $F \in \mathcal{F}_H$ of a subgraph $H$ of $G$ it holds that $q^{m}\NNM^{(H)}(\Phi_q = F)=\NNM^{(G)}(\Phi_q = F)$, where $m$ is the difference between the number of vertices in $G$ and $H$. This simplifies the notation required for various combinatorial manipulations.
\begin{lemma}[Graph extension lemma (single vertex version)]\label{lem:graph_extension_single_vertex}
 Let $G=(V,E,w)$ be a weighted directed graph and $x \in V$ a vertex. 
 Let $R_q$ be the set of root vertices of $\Phi_q$. Let $H=G[V\setminus \{x\}]$ denote the induced subgraph of $G$ obtained by removing vertex $x$.
 Let $\{\mathcal{H}(F) \c F \in \mathcal{F}_H\}$ be the partition of $\mathcal{F}_G$ given by 
 $\mathcal{H}(F) = \{F' \in \mathcal{F}_G \c F'[V \setminus \{x\}] = F\}$, 
 i.e. $\mathcal{H}(F)$ denotes the set of rooted spanning forests of 
 $G$ for which the induced subgraph obtained by removing $x$ equals $F$.
 For each vertex $y \in V \setminus \{x\}$ let $r_y(F)$ denote the unique root in $F$ that is connected to $y$.
 Then it holds for all $F \in \mathcal{F}_H$ that
\begin{equation*}
 \NNM^{(G)}(\Phi_q \in \mathcal{H}(F), \ x \in R_q)
 =q  \ \NNM^{(H)}(\Phi_q  = F) \ \prod_{r \in R(F)} (1+\tfrac{w(r,x)}{q})
\end{equation*}
 and that
\begin{equation*}
 \NNM^{(G)}(\Phi_q \in \mathcal{H}(F), \ x \notin R_q)
 =\NNM^{(H)}(\Phi_q  = F) \ \sum_{y \in V \setminus \{x\}} w(x,y)
 \prod_{r \in R(F) \setminus \{r_y(F)\}} (1+\tfrac{w(r,x)}{q}).
\end{equation*}
Here we take $w(e)=0$ when $e \notin E$.
\end{lemma}
\begin{proof}[Proof of \cref{lem:graph_extension_single_vertex}]
 We will first prove the first equality. 
 Let $F_H \in \mathcal{F}_H$ be given.
 Each forest in $F \in \mathcal{H}(F_H)$ with $x \in R(F)$ can be obtained from $F_H$ by adding any number of edges from roots of 
 $F_H$ to $x$. So, for each root we can choose either to add this edge or not to add this edge.
 For each edge we do add there will be one less component, since the root from which that edge originated will cease to be a root in the new forest. This contributes a factor $\tfrac{1}{q}$. We then also have an additional edge, which contributes a factor equal to the weight of that edge. This gives us the product over the roots $r$, where the $1$ term is chosen if no edge is added from $r$ to $x$ and the $\tfrac{w(r,x)}{q}$ term is chosen if we do add such an edge.
 If we don't add any such edges, then the obtained forest will have one more root than $F_H$, so this gives us the additional factor $q$. 
 
 The second equality is proven similarly.
 Each forest in $F \in \mathcal{H}(F_H)$ with $x \notin R(F)$ can be obtained from $F_H$ by first adding a single edge from $x$ to any other vertex $y$. We then add any number of edges from roots of $F_H$ to $x$, but we cannot add an edge from $r_y$ to $x$ as this would create a cycle. 
\end{proof}

\begin{definition}
Let $G=(V,E)$ be a directed graph. Let $A \subseteq V$ be a set of vertices and 
denote by $G[A]$ the induced subgraph of $G$ on the vertices in $A$.
A set $\mathcal{H} \subseteq \mathcal{F}$ of rooted forests of $G$ is said to be \emph{determined} by $A$ 
if there exists an $\mathcal{A} \subseteq \mathcal{F}_{G[A]}$ such that $\mathcal{H}=\cbrac{F \in \mathcal{F}_{G} \c F[A] \in \mathcal{A}}$.
\end{definition}

\begin{lemma}[Graph extension lemma (single edge version)]\label{lem:graph_extension_single_edge}
 Let $G=(V,E,w)$ be a weighted directed graph. Let $\cbrac{A,B}$ be a partition of $V$ and 
 assume that there exists exists exactly one vertex $a \in A$ that is adjacent to any vertices in $B$
 and exactly one vertex $b \in B$ adjacent to any vertices in $A$.
 Write $G[A]$ and $G[B]$ to denote the induced subgraphs on $A$ and $B$.
 Let $\mathcal{A}, \mathcal{B} \subseteq \mathcal{F}$ be sets of rooted forests of $G$ that are 
 determined by $A$ and $B$, respectively, and let $\mathcal{A}' \subseteq \mathcal{F}_{G[A]}$ and $\mathcal{B}' \subseteq \mathcal{F}_{G[B]}$ be such that $\mathcal{A}=\cbrac{F \in \mathcal{F}_{G} \c F[A] \in \mathcal{A}'}$ and $\mathcal{B}=\cbrac{F \in \mathcal{F}_{G} \c F[B] \in \mathcal{B}'}$. Denote by $R_q$ the set of root vertices of $\Phi_q$.
 Then it holds that
 \begin{align*}
  \NNM^{(G)}(\Phi_q \in \mathcal{A} \cap \mathcal{B}) &= 
  \NNM^{(G[A])}(\Phi_q \in\mathcal{A}') \ \NNM^{(G[B])}(\Phi_q \in \mathcal{B}') \\
  &\quad+ \tfrac{w(a,b)}{q} \NNM^{(G[B])}(\Phi_q \in \mathcal{B}') \ \NNM^{(G[A])}(\Phi_q \in \mathcal{A}', \ a \in R_q) \\
  &\quad+\tfrac{w(b,a)}{q}\NNM^{(G[A])}(\Phi_q \in \mathcal{A}') \ \NNM^{(G[B])}(\Phi_q \in \mathcal{B}', \ b \in R_q).
 \end{align*}
\end{lemma}
\begin{proof}
 Let $F_A \in \mathcal{A}'$ and $F_B \in \mathcal{B}'$ be given.
 
 If both $a$ is a root in $F_A$ and $b$ is a root in $F_B$, 
 then there are exactly three forests $F_1, F_2, F_3 \in \mathcal{F}_{G}$ 
 for which the induced subgraphs on $A$ and $B$ correspond to $F_A$ and $F_B$, respectively.
 \begin{enumerate}
  \item The first of these forests consists of the disjoint graph union of $F_A$ and $F_B$ and has non-normalized measure 
 $$\NNM^{(G)}(\Phi_q  = F_1)=\NNM^{(G[A])}(\Phi_q  = F_A) \ \NNM^{(G[B])}(\Phi_q  = F_B).$$
 \item The second has an additional edge from $a$ to $b$ and has non-normalized measure 
 $$\NNM^{(G)}(\Phi_q  = F_2)=\tfrac{w(a,b)}{q}\NNM^{(G[A])}(\Phi_q  = F_A) \ \NNM^{(G[B])}(\Phi_q  = F_B),$$
 since it contains one less root than the sum of the roots in $F_A$ and $F_B$ and one additional edge with weight $w(a,b)$.
 \item The third forest has an additional edge from $b$ to $a$ and it similarly has non-normalized measure 
 $$\NNM^{(G)}(\Phi_q  = F_3)=\tfrac{w(b,a)}{q}\NNM^{(G[A])}(\Phi_q  = F_A) \ \NNM^{(G[B])}(\Phi_q  = F_B).$$
 \end{enumerate}
 Note that each of these three forests is contained in $\mathcal{A} \cap \mathcal{B}$.
 
 If exactly one of the vertices $a$ and $b$ is a root in $F_A$ and $F_B$, then only two of the above mentioned forests are rooted forest of $G$, since adding an outgoing edge to a non-root vertex does not yield a rooted forest.
 
 If both $a$ and $b$ are not roots, then only the first forest without an additional edge is a rooted forest of $G$.
 
 Since each forest in $\mathcal{A} \cap \mathcal{B}$ can be obtained in such a manner, 
 summing over all rooted forests in $\mathcal{A}'$ and $\mathcal{B}'$ yields
 \begin{align*}
  \NNM^{(G)}_q(\mathcal{A} \cap \mathcal{B}) &= 
  \sum_{F_A \in \mathcal{A}'} \sum_{F_B \in \mathcal{B}'} 
  \NNM^{(G[A])}(\Phi_q  = F_A)\NNM^{(G[B])}(\Phi_q  = F_B) \\
  &\hspace{9 em}+\tfrac{w(a,b)}{q}\NNM^{(G[A])}(\Phi_q  = F_A)\NNM^{(G[B])}(\Phi_q  = F_B) \mathbf{1}_{\cbrac{b \in R(F_A)}} \\
  &\hspace{9 em}+\tfrac{w(b,a)}{q}\NNM^{(G[A])}(\Phi_q  = F_A)\NNM^{(G[B])}(\Phi_q  = F_B) \mathbf{1}_{\cbrac{b' \in R(F_B)}} \\
  &= \NNM^{(G[A])}(\Phi_q \in\mathcal{A}') \ \NNM^{(G[B])}(\Phi_q \in \mathcal{B}') \\
  &\quad+ \tfrac{w(a,b)}{q}\NNM^{(G[A])}(\Phi_q \in \mathcal{A}', \ a \in R_q) \ \NNM^{(G[B])}(\Phi_q \in \mathcal{B}') \\
  &\quad+\tfrac{w(b,a)}{q}\NNM^{(G[A])}(\Phi_q \in \mathcal{A}') \ \NNM^{(G[B])}(\Phi_q \in \mathcal{B}', \ b \in R_q).
 \end{align*}
\end{proof}


\subsection{Monotonicities on undirected networks: proof of Theorem
\ref{thm:monotone_correlations_symmetric}}\label{ssec:monotonicity_proofs_symmetric}

\begin{lemma}\label{thm:increasing_roots_edge_removal}
 Let $G=(V,E,w)$ be a weighted symmetric graph and let $B \subseteq E$ be a symmetric subset of directed edges, i.e. $(x,y) \in B \implies (y,x) \in B$.
 Then for all $q>0$ it holds that
 \begin{equation*}
 \E[r_q \mid B \cap \Phi_q=\varnothing] \geq \E[r_q]. 
 \end{equation*}
\end{lemma}
\begin{proof}
  Let $H=G-B$ denote the subgraph of $G$ obtained by removing all edges in $B$.
  Let $L^{(G)}$ and $L^{(H)}$ denote the graph Laplacians of $G$ and $H$, respectively. 
  Since these Laplacians are symmetric, $-L^{(G)}$ and $-L^{(H)}$ have real eigenvalues $\lambda_n \geq\ldots\geq\lambda_1$ and $\mu_n \geq\ldots\geq\mu_1$, respectively.
  By Weyl's monotonicity principle, these eigenvalues satisfy 
  $\lambda_i \geq \mu_i$ for all $i \in [n]$.
  It follows that $\mathrm{Tr}((qI-L^{(G)})^{-1}) \leq \mathrm{Tr}((qI-L^{(H)})^{-1})$.
  By the spatial Markov property and \cite[prop~2.1]{avena2018applications} it then holds that
  \begin{equation*}
    \E^{(G)}[r_q \mid B \cap \Phi_q=\varnothing] 
    = \E^{(H)}[r_q] 
    = q\mathrm{Tr}((qI-L^{(H)})^{-1}) 
    \geq q\mathrm{Tr}((qI-L^{(G)})^{-1}) 
    = \E^{(G)}[r_q].
  \end{equation*}
\end{proof}

\begin{lemma}\label{thm:decreasing_roots_edge_conditioning}
 Let $G=(V,E,w)$ be a be a weighted symmetric graph and $A \subseteq E$.
 If $\P(\pm A \subseteq \Phi_q)>0$, then for all $q>0$ it holds that
 \begin{equation*}
 \E[r_q \mid \pm A \subseteq \Phi_q] \leq \E[r_q]. 
 \end{equation*}
\end{lemma}
\begin{proof}
  Let $e=(x,y) \in A$ be given. By \cref{lem:spatial_markov_property} and \cref{thm:increasing_roots_edge_removal} it holds that
  \begin{align*}
    \E[r_q \mid \pm A \subseteq \Phi_q] 
    &= \frac{\E[r_q \mid \pm(A-e) \subseteq \Phi_q]
    -\E[r_q \mid  \pm(A-e) \subseteq \Phi_q, \pm e \notin \Phi_q]\P(\pm e \notin \Phi_q\mid \pm(A-e) \subseteq \Phi_q)}{\P(\pm e \in \Phi_q\mid \pm(A-e) \subseteq \Phi_q)} \\
     &= \frac{\E[r_q \mid \pm(A-e) \subseteq \Phi_q]
    -\E^{(G/(A-e))}[r_q \mid \pm e \notin \Phi_q]\P(\pm e \notin \Phi_q\mid \pm(A-e) \subseteq \Phi_q)}{\P(\pm e \in \Phi_q\mid \pm(A-e) \subseteq \Phi_q)} \\
    &\leq \frac{\E[r_q \mid \pm(A-e) \subseteq \Phi_q]
    -\E^{(G/(A-e))}[r_q]\P(\pm e \notin \Phi_q\mid \pm(A-e) \subseteq \Phi_q)}{\P(\pm e \in \Phi_q\mid \pm(A-e) \subseteq \Phi_q)} \\
    &= \frac{\E[r_q \mid \pm(A-e) \subseteq \Phi_q]
    -\E[r_q  \mid \pm(A-e) \subseteq \Phi_q]\P(\pm e \notin \Phi_q\mid \pm(A-e) \subseteq \Phi_q)}{\P(\pm e \in \Phi_q\mid \pm(A-e) \subseteq \Phi_q)} \\
    &= \E[r_q \mid \pm(A-e) \subseteq \Phi_q].
  \end{align*}
Hence, the result follows by induction on $\abs{A}$.
\end{proof}
 
\begin{proof}[Proof of \cref{thm:monotone_correlations_symmetric}]
The proof of~\eqref{MonEdgeSet} follows directly from \cref{thm:decreasing_roots_edge_conditioning,thm:derivative_probabilities}.

For the statement about the pairwise LEP potential in~\eqref{MonPotential} we argue as follows.
Fix $q >0$. By \cref{thm:derivative_probabilities} it is sufficient to show that 
  $\E[r_q] \geq \E[r_q \mid B_q(x)=B_q(y)]$.
  
  Let $\mathcal{P}$ denote the set of undirected paths from $x$ to $y$, where we interpret a path as a set of directed edges. Then the event 
  $\cbrac{B_q(x)=B_q(y)}$ can be written as the disjoint union
  \begin{equation*}
   \cbrac{B_q(x)=B_q(y)} = \bigcup_{\pi \in \mathcal{P}} \cbrac{\pm \pi \subseteq \Phi_q}. 
  \end{equation*}
It follows by \cref{thm:decreasing_roots_edge_conditioning} that
\begin{align*}
  \E[r_q \mid B_q(x)=B_q(y)] 
  &= \sum_{\pi \in \mathcal{P}} \E[r_q \mid \pm \pi \subseteq \Phi_q]\P(\pm \pi \subseteq \Phi_q \mid B_q(x)=B_q(y)) \\
  & \leq \sum_{\pi \in \mathcal{P}} \E[r_q]\P(\pm  \pi \subseteq \Phi_q \mid B_q(x)=B_q(y)) = \E[r_q].
\end{align*}
\end{proof}

\section{Two-points correlations on trees}
\subsection{Monotonicity of correlations on general trees }\label{ssec:monotonicity_proofs_trees}

Below we show the monotonicity of the 2-point correlation restricted to arbitrary trees. We will start by expressing the 2-point correlation via hitting times in \cref{lem:hitting_time_expression}. Then in \cref{prop:monotone_rooting} we then show the monotocity of one point rooting events, by means of \cref{thm:derivative_probabilities}. After a last  
bound on the derivatives of hitting time events \cref{lem:derivative_bound_conditional_rooting},
we derive the main claim using these three lemmas.

\begin{lemma}[Hitting time expression for two-point correlation between adjacent vertices in trees]\label{lem:hitting_time_expression}
 Let $G=(V,E,w)$ be a weighted directed tree and $x,y \in V$ two adjacent vertices.
 Let $\P_v$ denote the law of the random walk $X$ starting at vertex $v \in V$.
 The hitting time of vertex $v$ by $X$ is denoted by $\tau_v$ and $\tau_{q}$ is an independent exponential killing time with rate $q$.
 Then it holds that 
\begin{align*}
 \correlation[G]{q}{x}{y}
 =\frac{1-\P_x(\tau_y < \tau_{q})-\P_y(\tau_x < \tau_{q})+\P_x(\tau_y < \tau_{q})\P_y(\tau_x < \tau_{q})}  
 {1-\P_x(\tau_y < \tau_{q})\P_y(\tau_x < \tau_{q})}.
\end{align*}
\end{lemma}

\begin{proof}[Proof of \cref{lem:hitting_time_expression}]
We will reason using the representation in~\eqref{LEdec} coming from Wilson's sampling construction.
We note in particular that in order for the directed edge $(x,y)$ to be present in $\Phi_q$, it is equivalent to require that the loop-erased trajectory in~\eqref{LEdec} includes $y$, which can be expressed in terms of hitting times of the random walk as 
\begin{align*}
 \P((x,y) \in \Phi_q) &= \P_x(\tau_y < \tau_{q})
 \sum_{k=0}^{\infty}\paren*{\P_y(\tau_x < \tau_{q})\P_x(\tau_y < \tau_{q})}^k\P_y(\tau_{q} < \tau_x) \\
 &=\frac{\P_x(\tau_y < \tau_{q})\paren*{1-\P_y(\tau_x < \tau_{q})}}{1-\P_x(\tau_y < \tau_{q})\P_y(\tau_x < \tau_{q})}, 
\end{align*}
where the index $k$ in the above sum represents the number of times that the random walk reaches $y$ and then does return to $x$. We notice in particular that the above step is equivalent to use the forest transfer-current kernel in~\cite{avena2018transfercurrent}.

For the reversed edge $(y,x)$, we can write
\begin{align*}
 \P((y,x) \in \Phi_q) &= \paren*{1-\P((x,y) \in \Phi_q)}\P_y(\tau_x < \tau_{q}),
\end{align*}
where these two factors correspond to~\eqref{LEdec}. 
Therefore, it follows that
\begin{align*}
 \correlation[G]{q}{x}{y} &=  1-\P((x,y) \in \Phi_q)-\P((y,x) \in \Phi_q) \\
 &=1-\tfrac{\P_x(\tau_y < \tau_{q})\paren*{1-\P_y(\tau_x < \tau_{q})}}{1-\P_x(\tau_y < \tau_{q})\P_y(\tau_x < \tau_{q})}
 -\paren*{1-\tfrac{\P_x(\tau_y < \tau_{q})\paren*{1-\P_y(\tau_x < \tau_{q})}}{1-\P_x(\tau_y < \tau_{q})\P_y(\tau_x < \tau_{q})}}\P_y(\tau_x < \tau_{q}) \\ 
 &=\frac{1-\P_x(\tau_y < \tau_{q})-\P_y(\tau_x < \tau_{q})+\P_x(\tau_y < \tau_{q})\P_y(\tau_x < \tau_{q})}{1-\P_x(\tau_y < \tau_{q})\P_y(\tau_x < \tau_{q})}.
\end{align*}
\end{proof}

\begin{lemma}[Bound on derivative of hitting probabilities]\label{lem:derivative_bound_hitting}
Let $G=(V,E,w)$ be a weighted directed graph and $x,y \in V$ two vertices.
Let $\P^{(q)}_x$ denote the law of the random walk $X$ on  $G$ starting at $x$.
For each $v \in V$ let $\tau_v$ denote the hitting time of $v$ by $X$ and let $\tau_q$ be 
an independent exponential killing time with rate $q$.
Then it holds for the derivative of the function $q \mapsto \P_x(\tau_{y} < \tau_{q})$ that 
\begin{equation}\label{1}
 \tfrac{1}{q}\P_x(\tau_{y} < \tau_{q})-\tfrac{1}{q} \leq \frac{d}{d q}\P_x(\tau_{y} < \tau_{q}) \leq 0.
\end{equation}
\end{lemma}

In the subsequent proofs it will be convenient to work with the discrete-time skeleton of the random walk $X$, that is, the discrete-time random walk $\tilde{X}$ on $G$ with transition matrix 
\begin{equation}\label{eq:discrete_time_rw}
P=I+\tfrac{1}{\alpha}L,
\end{equation}
with $\alpha$ the maximal diagonal entry of the negative graph Laplacian $-L$.
The path measure of $\tilde{X}$ starting at $x$ is denoted by $\tilde{\P}_x$. For $\tilde{\tau}_{q}$ an independent ($\N$-valued) geometric killing time with success probability $\tfrac{q}{q+\alpha}$, it then holds that $\P_x(\tau_{y} < \tau_{q})=\tilde{\P}_x(\tau_{y} < \tilde{\tau}_{q})$. Since the law of the loop-erased trajectory of $\tilde{X}$ corresponds to that of $X$, we can also use this discrete-time random walk to analyze \cref{LEdec}.

\begin{proof}[Proof of \cref{lem:derivative_bound_hitting}]
The upper bound on the derivative in~\eqref{1} is immediate, we therefore show the lower bound.

Let $\tilde{\P}_x$ denote the law of the discrete-time random walk $\tilde{X}$, as defined in \cref{eq:discrete_time_rw}. Then it holds that 
\begin{align*}
\P_x(\tau_{y} < \tau_{q}) &= \tilde{\P}_x(\tau_{y} < \tilde{\tau}_{q}) = \sum_{k=1}^{\infty}\tilde{\P}_x(\tau_{y} < k)\P(\tilde{\tau}_{q}=k) 
= \sum_{k=1}^{\infty}\tilde{\P}_x(\tau_{y} < k)\tfrac{q}{q+\alpha}\paren*{1-\tfrac{q}{q+\alpha}}^{k-1}.
\end{align*}
Since $\tilde{\P}_x(\tau_{y} < k)$ does not depend on $q$, it follows that
 \begin{align*}
 \frac{d}{d q}\P_x(\tau_{y} < \tau_{q})
 &=\sum_{k=1}^{\infty}\tilde{\P}_x(\tau_{y} < k) \frac{(1-k)q+\alpha}{(q+\alpha)^2}\paren*{\frac{\alpha}{q+\alpha}}^{k-1} =\tfrac{1}{q}\P_x(\tau_{y} < \tau_{q})-\sum_{k=1}^{\infty}\tilde{\P}_x(\tau_{y} < k)\frac{kq}{(q+\alpha)^2}\paren*{\frac{\alpha}{q+\alpha}}^{k-1} \\
 &\geq \tfrac{1}{q}\P_x(\tau_{y} < \tau_{q})
 -\sum_{k=1}^{\infty}\frac{kq}{(q+\alpha)^2}
 \paren*{\frac{\alpha}{q+\alpha}}^{k-1} 
 =\tfrac{1}{q}\P_x(\tau_{y} < \tau_{q})-\tfrac{1}{q+\alpha}\E(\tilde{\tau}_{q}) 
 =\tfrac{1}{q}\P_x(\tau_{y} < \tau_{q})-\tfrac{1}{q}.
 \end{align*}
\end{proof}

\begin{lemma}[Monotonicity of rooting probabilities] \label{prop:monotone_rooting}
 Let $G=(V,E,w)$ be a weighted directed graph and $x \in V$ a vertex. 
 Let $\P$ denote the law of a random rooted spanning forest $\Phi_q$ of $G$ with rooting parameter $q>0$.
 Let $R_q$ denote the set of roots of $\Phi_q$.
 Then it holds that
 \begin{equation}
  0 \leq \frac{d}{dq}\P(x \in R_q) \leq \tfrac{1}{q}\P(x \in R_q).
 \end{equation}
\end{lemma}

\begin{proof}[Proof of \cref{prop:monotone_rooting} via \cref{lem:derivative_bound_hitting}]
Due to the determinantality of the roots in~\eqref{DetR}, we have that $x$ is a root in $\Phi_q$ if 
\begin{align*}
 \P(x \in R_q) &=K_q(x,x)=q(q-L)^{-1}(x,x)=\P_x(X_{\tau_{q}} = x).
\end{align*}
Let $N_x$ denote the set of out-neighbours of $x$ in $G$.
Let $\sigma=\inf\cbrac{t > 0 \c X_t \neq X_0}$ be the first jump time of $X$.
Then by the Markov property of $X$ we have that 
\begin{align*}
 \P_x(X_{\tau_{q}} = x) &= \P_x(\sigma > \tau_q)
 + \sum_{v \in N_x}\P_x(X_{\sigma} = v)\P_v(\tau_{x} < \tau_{q})\P_x(X_{\tau_{q}} = x).
\end{align*}
Solving this equation gives us that
\begin{align*}
\P_x(X_{\tau_{q}} = x) &= \frac{\P_x(\sigma > \tau_q)}{1-\sum_{v \in N_x}\P_x(X_{\sigma} = v)\P_v(\tau_{x} < \tau_{q})} 
= \frac{q}{q+\sum_{v \in N_x}w(x,v)(1-\P_v(\tau_{x} < \tau_{q}))}.
\end{align*}
It follows by \cref{lem:derivative_bound_hitting} that
\begin{align*}
\frac{d}{d q} \P(x \in R_q) &=  \frac{d}{d q}\P_x(X_{\tau_{q}} = x) = \frac{\sum_{v \in N_x}w(x,v)\paren*{1-\P_v(\tau_{x} < \tau_{q})+q\tfrac{d}{d q}\P_v(\tau_{x} < \tau_{q}))}}
 {\paren*{q+\sum_{v \in N_x}w(x,v)(1-\P_v(\tau_{x} < \tau_{q}))}^2}  \geq 0,
\end{align*}
which proves the lower bound.
For the upper bound it holds that
\begin{align*}
\frac{\frac{d}{d q} \P(x \in R_q)} {\P(x \in R_q)} 
&= \frac{\sum_{v \in N_x}w(x,v)\paren*{1-\P_v(\tau_{x} < \tau_{q})+q\tfrac{d}{d q}\P_v(\tau_{x} < \tau_{q}))}}
{q\paren*{q+\sum_{v \in N_x}w(x,v)(1-\P_v(\tau_{x} < \tau_{q}))}} \\
&\leq \frac{\sum_{v \in N_x}w(x,v)\paren*{1-\P_v(\tau_{x} < \tau_{q})}}
{q\paren*{q+\sum_{v \in N_x}w(x,v)(1-\P_v(\tau_{x} < \tau_{q}))}} \leq \frac{1}{q}.
\end{align*}
\end{proof}

\begin{lemma}[Bound on conditional rooting derivative in trees] \label{lem:derivative_bound_conditional_rooting}
 Let $G=(V,E,w)$ be a weighted directed tree and $x,y \in V$ two vertices.
 Then it holds that
 \begin{equation}
 \tfrac{d}{dq} \P(x \in R_q \mid x \leftrightarrow y) \leq \tfrac{1}{q} \P(x \in R_q \mid x \leftrightarrow y).
 \end{equation}
\end{lemma}
\begin{proof}[Proof of \cref{lem:derivative_bound_conditional_rooting}]
 Let $d$ denote the distance between $x$ and $y$. We will argue inductively on $d$.
 For $d=0$ the statement follows from \cref{prop:monotone_rooting}.
 
 Now assume that $d \geq 1$. Let $z$ denote the vertex adjacent to $x$ with distance $d-1$ to $y$. 
 Note that we possibly have that $z=y$. 
 Since $G$ is a tree, removing the edges between $x$ and $z$ splits the graph into two components $T_x$ and $T_z$,
 where $T_x$ and $T_z$ denote the component containing vertex $x$ and $z$, respectively.
 It then holds by \cref{lem:graph_extension_single_edge} that
 \begin{align*}
  \P(x \in R_q \mid x \leftrightarrow y) &= 
  \frac{w(z,x)\NNM^{(T_x)}(x \in R_q)\NNM^{(T_z)}(z \in R_q, z \leftrightarrow y)}
  {w(z,x)Z_{T_x}(q)\NNM^{(T_z)}(z \in R_q, z \leftrightarrow y)+w(x,z)\NNM^{(T_x)}(x \in R_q)\NNM^{(T_z)}(z \leftrightarrow y)} \\
  &= 
  \frac{w(z,x)\P^{(T_z)}(z \in R_q \mid z \leftrightarrow y)\P^{(T_x)}(x \in R_q)}
  {w(z,x)\P^{(T_z)}(z \in R_q \mid z \leftrightarrow y)+w(x,z)\P^{(T_x)}(x \in R_q)}.
 \end{align*}
It follows by the induction hypothesis and \cref{prop:monotone_rooting} that
\begin{align*}
  &\frac{\frac{d}{dq}\P(x \in R_q \mid x \leftrightarrow y)}{\P(x \in R_q \mid x \leftrightarrow y)} \\
  &\quad = 
  \frac{w(z,x)\P^{(T_x)}(x \in R_q)^2 \tfrac{d}{dq} \P^{(T_z)}(z \in R_q \mid z \leftrightarrow y)
  +w(x,z)\P^{(T_z)}(z \in R_q \mid z \leftrightarrow y)^2 \tfrac{d}{dq} \P^{(T_x)}(x \in R_q)}
  {w(z,x)\P^{(T_z)}(z \in R_q \mid z \leftrightarrow y)\P^{(T_x)}(x \in R_q)^2
  +w(x,z)\P^{(T_z)}(z \in R_q \mid z \leftrightarrow y)^2\P^{(T_x)}(x \in R_q)} \\
  &\quad \leq \tfrac{1}{q}.
 \end{align*}
\end{proof}

\begin{proof}[Proof of \cref{thm:monotone_correlations_trees}]
Let $d=d(x,y)$ denote the distance between $x$ and $y$ in $G$ and let $z$ be the vertex adjacent to $x$ with distance $d-1$ to $y$.
We proceed by induction on $d$.

If $d=1$, then by \cref{lem:hitting_time_expression} we have that
\begin{align*}
 \correlation{q}{x}{y}=\frac{1-\P_x(\tau_y < \tau_q)-\P_y(\tau_x < \tau_q)+\P_x(\tau_y < \tau_q)\P_y(\tau_x < \tau_q)}
 {1-\P_x(\tau_y < \tau_q)\P_y(\tau_x < \tau_q)}.
\end{align*}
Taking the derivative gives us that
\begin{align*}
 \frac{d}{dq}\correlation{q}{x}{y}=
 -\frac{(1-\P_x(\tau_y < \tau_q))^2\tfrac{d}{dq}\P_y(\tau_x < \tau_q)
 +(1-\P_y(\tau_x < \tau_q))^2\tfrac{d}{dq}\P_x(\tau_y < \tau_q)}
 {\paren*{1+\P_x(\tau_y < \tau_q)\P_y(\tau_x < \tau_q)}^2},
\end{align*}
which is non-negative by the upper bound in \cref{lem:derivative_bound_hitting}.

Now assume that $d \geq 2$. We then have that
\begin{align*}
\frac{d}{dq} \correlation{q}{x}{y}
&= \frac{d}{dq} \paren*{\P(x \nleftrightarrow z \mid z \leftrightarrow y)\P(z \leftrightarrow y)+\P(z \nleftrightarrow y)} \\
&= \P(x \nleftrightarrow z \mid z \leftrightarrow y)\tfrac{d}{dq} \P(z \leftrightarrow y) 
+ \P(z \leftrightarrow y) \tfrac{d}{dq} \P(x \nleftrightarrow z \mid z \leftrightarrow y) 
+ \tfrac{d}{dq} \P(z \nleftrightarrow y) \\
&= (1-\P(x \nleftrightarrow z \mid z \leftrightarrow y))\tfrac{d}{dq} \P(z \nleftrightarrow y) 
+ \P(z \leftrightarrow y) \tfrac{d}{dq} \P(x \nleftrightarrow z \mid z \leftrightarrow y).
\end{align*}
By the induction hypothesis, we have that $\tfrac{d}{dq} \P(z \nleftrightarrow y) \geq 0$.
Hence, it remains to show that $\tfrac{d}{dq} \P(x \nleftrightarrow z \mid z \leftrightarrow y) \geq 0$.

Removing the edges between $x$ and $z$ splits $G$ into two connected components.
Let $T_x$ and $T_z$ denote the components containing vertex $x$ and $z$, respectively.
By \cref{lem:graph_extension_single_edge} it then holds that
\begin{align*}
 \P(x \nleftrightarrow z \mid z \leftrightarrow y) &= 
 \frac{Z_{T_x}(q) \NNM^{(T_{z})}(z \leftrightarrow y)}
 {Z_{T_x}(q) \NNM^{(T_{z})}(z \leftrightarrow y) 
 + \tfrac{w(x,z)}{q} \NNM^{(T_{x})}(x \in R_q) \NNM^{(T_{z})}(z \leftrightarrow y)
 + \tfrac{w(z,y)}{q} Z_{T_x}(q) \NNM^{(T_{z})}(z \in R_q, \ z \leftrightarrow y)} \\
 &=  \frac{q}{q + w(x,z) \P^{(T_{x})}(x \in R_q) + w(z,y) \P^{(T_{z})}(z \in R_q \mid z \leftrightarrow y)}.
\end{align*}
Taking the derivative and applying \cref{prop:monotone_rooting,lem:derivative_bound_conditional_rooting} gives us that
\begin{align*}
 \frac{d}{dq}  \P(x \nleftrightarrow z \mid z \leftrightarrow y) &=  \frac{w(x,z) \P^{(T_{x})}(x \in R_q) + w(z,y) \P^{(T_{z})}(z \in R_q \mid z \leftrightarrow y)}{
 {\paren*{q + w(x,z) \P^{(T_{x})}(x \in R_q) + w(z,y) \P^{(T_{z})}(z \in R_q \mid z \leftrightarrow y)}^2}}
\\
&\quad - \frac{qw(x,z) \tfrac{d}{dq} \P^{(T_{x})}(x \in R_q) + qw(z,y) \tfrac{d}{dq} \P^{(T_{z})}(z \in R_q \mid z \leftrightarrow y)}{\paren*{q + w(x,z) \P^{(T_{x})}(x \in R_q) + w(z,y) \P^{(T_{z})}(z \in R_q \mid z \leftrightarrow y)}^2} \geq  0.
\end{align*}
\end{proof}

\subsection{Inclusion-exclusion for pairwise LEP-interaction potential on general trees }
\begin{proof}[Proof of \cref{prop:inex}]
We will prove the statement by induction on $d$. 
First assume that $d=1$. Write $\mathcal{H}=\cbrac{F \in \mathcal{F}_G \c x \nleftrightarrow_{F} y}$ to denote the set of rooted forests not containing an edge between $x$ and $y$. Since $G$ is a tree, removing the edges between $x$ and $y$ yields two connected components $G_x$ and $G_y$
containing vertex $x$ and vertex $y$, respectively.
Note that for the non-normalized measure on $G$ it holds that $\NNM^{(G)}(\Phi_q = F)=\NNM^{(G_x)}_q(\Phi = F[G_x])\NNM^{(G_y)}_q(\Phi = F[G_y])$ for all $F \in \mathcal{H}$,
where $F[G_x]$ and $F[G_x]$ denote the induced subgraphs of $F$ on the vertices of $G_x$ and $G_y$, respectively. 
For all $F_x \in \mathcal{F}_{G_x}$ and $F_y \in \mathcal{F}_{G_y}$ there is exactly one $F \in \mathcal{H}$ with $F[G_x]=F_x$ and $F[G_y]=F_y$, namely the disjoint graph union of $F_x$ and $F_y$.
Hence, it holds that
\begin{align*}
 U^{(G)}_q(x,y) &=  \sum_{F \in \mathcal{H}} \P^{(G)}(\Phi_q = F) = \frac{1}{Z_G(q)} \sum_{F \in \mathcal{H}} \NNM^{(G_x)}(\Phi_q = F[G_x])\NNM^{(G_y)}(\Phi_q = F[G_y]) \\
 &= \frac{1}{Z_G(q)} \sum_{F_x \in \mathcal{F}_{G_x}}\sum_{F_y \in \mathcal{F}_{G_y}} \NNM^{(G_x)}(\Phi_q = F_x)\NNM^{(G_y)}(\Phi_q = F_y) =\frac{Z_{G_x}(q)Z_{G_y}(q)}{Z_G(q)}.
\end{align*}

Now assume that $d > 1$. Let $z=z_{d-1}$ denote the neighbor of $y$ with distance $d-1$ to $x$.
Let $G_{\cbrac{d}}$ denote the graph obtained from $G$ by removing the edges between $y$ and $z$.
Then $G_{\cbrac{d}}$ consists of two components $G_y$ and $G_z$ containing vertex $y$ and $z$ respectively.
It then holds by \cref{lem:spatial_markov_property} and the induction hypothesis that
\begin{align*}
 U^{(G)}_q(x,y) &= U^{(G)}_q(x,z)+U^{(G)}_q(y,z)-\P^{(G)}(x \nleftrightarrow_{\Phi_q} z, \ y \nleftrightarrow_{\Phi_q} z) \\
 &= U^{(G)}_q(x,z)+U^{(G)}_q(y,z)-U^{(G)}_q(y,z)\P^{(G)}(x \nleftrightarrow_{\Phi_q} z \mid y \nleftrightarrow_{\Phi_q} z) \\
 &= U^{(G)}_q(x,z)+U^{(G)}_q(y,z)-U^{(G)}_q(y,z)U^{(G_{\cbrac{d}})}_q(x,z) \\
 &= \frac{1}{Z_G(q)} \paren*{\sum_{k=1}^{d-1}(-1)^{k+1} \sum_{I \in \binom{[d-1]}{k}} \prod_{i=1}^{k+1} Z_{G_{I}^{i}}(q)} 
 + \frac{Z_{G_y}(q)Z_{G_z}(q)}{Z_G(q)} \\
 &\quad-\frac{Z_{G_y}(q)Z_{G_z}(q)}{Z_G(q)}\frac{1}{Z_{G_{\cbrac{d}}}(q)}
 \paren*{\sum_{k=1}^{d-1}(-1)^{k+1} \sum_{I \in \binom{[d-1]}{k}} \prod_{i=1}^{k+2} Z_{G_{I\cup\cbrac{d}}^{i}}(q)} \\
 &= \frac{1}{Z_G(q)} \paren*{\sum_{k=1}^{d-1}(-1)^{k+1} \sum_{I \in \binom{[d-1]}{k}} \prod_{i=1}^{k+1} Z_{G_{I}^{i}}(q)} 
 + \frac{Z_{G_y}(q)Z_{G_z}(q)}{Z_G(q)} \\
 &\quad+\frac{1}{Z_G(q)}
 \paren*{\sum_{k=1}^{d-1}(-1)^{k} \sum_{I \in \binom{[d-1]}{k}} \prod_{i=1}^{k+2} Z_{G_{I\cup\cbrac{d}}^{i}}(q)} \\
 &= \frac{1}{Z_G(q)} \paren*{\sum_{k=1}^{d-1}(-1)^{k+1} \sum_{I \in \binom{[d-1]}{k}} \prod_{i=1}^{k+1} Z_{G_{I}^{i}}(q)} 
 +\frac{1}{Z_G(q)}
 \paren*{\sum_{k=0}^{d-1}(-1)^{k} \sum_{I \in \binom{[d-1]}{k}} \prod_{i=1}^{k+2} Z_{G_{I\cup\cbrac{d}}^{i}}(q)} \\
 &= \frac{1}{Z_G(q)} \paren*{\sum_{k=1}^{d}(-1)^{k+1} \sum_{I \in \binom{[d]}{k}} \prod_{i=1}^{k+1} Z_{G_{I}^{i}}(q)}.
\end{align*}
\end{proof}


\subsection{Partition function on segments and rings}

\begin{proof}[Proof of \cref{thm:path_partition}]
\leavevmode \\
\fbox{\cref{eq:path_partition_combinatorial}}
Let $b$ be a boundary vertex of $PG_n$.
Let $\NNM^{(n)}$ denote the non-normalized measure on $PG_n$.
By \cref{lem:graph_extension_single_edge} we have that
\begin{equation}\label{eq:path_rooting_extension}
 \NNM^{(n)}(b \notin R_q) = Z_{n-1}(q), \text{ and } \ \NNM^{(n)}(b \in R_q) = \NNM^{(n-1)}(b \in R_q)+qZ_{n-1}(q).
\end{equation}
This gives us that
\begin{align}\notag
 Z_{n}(q) &= \NNM^{(n)}(b \in R_q) + \NNM^{(n)}(b \notin R_q)= \NNM^{(n-1)}(b \in R_q)+(q+1)Z_{n-1}(q)\\& = (q+2)Z_{n-1}(q)-\NNM^{(n-1)}(b \notin R_q) 
 = (q+2)Z_{n-1}(q)-Z_{n-2}(q). \label{eq:path_laplacian_recurrence}
\end{align}

We will prove \cref{eq:path_partition_combinatorial} by induction on $n$.
Note that for $n=1$ we have $Z_1(q)=q$ and for $n=2$ we have $Z_2(q)=q^2+2q$, so in both these cases \cref{eq:path_partition_combinatorial} holds. Now assume that $n>2$. 
Then by \cref{eq:path_laplacian_recurrence}, the induction hypothesis and repeated applications of Pascal's formula we have that
\begin{align*}
 Z_n(q) &= (2+q)Z_{n-1}(q)-Z_{n-2}(q)
 = (2+q)\sum_{k=1}^{n-1}\binom{n+k-2}{2k-1}q^k-\sum_{k=1}^{n-2}\binom{n+k-3}{2k-1}q^k 
 =\sum_{k=1}^{n}\binom{n+k-1}{2k-1}q^k.
\end{align*}

\noindent\fbox{\cref{eq:path_partition_spectral}}
Let $L$ denote the graph Laplacian of $PG_n$, since due to~\eqref{Z} the partition function is the characteristic polynomial of $L$, it  can be directly obtained from its spectrum, 
which is given in \cite{vanmieghem2010spectra}, from which: 
\begin{equation*}
 Z_n(q) = \prod_{k=1}^{n}\paren*{q+2-2\cos\paren*{\tfrac{\pi(n-k)}{n}}}.
\end{equation*}

\noindent\fbox{\cref{eq:path_partition_recurrence}}
We have shown above that the partition function satisfies the recurrence relation in \cref{eq:path_laplacian_recurrence} .
Using the initial conditions $Z_1(q)=q$ and $Z_2(q)=q^2+2q$, this linear recurrence relation has solution
\begin{equation*}
 Z_n(q) = \frac{q\paren*{q+2+\sqrt{q^2+4q}}^{n}-q\paren*{q+2-\sqrt{q^2+4q}}^{n}}{2^{n}\sqrt{q^2+4q}}.
\end{equation*}

\noindent\fbox{\cref{eq:path_partition_chebyshev}}
To verify that the three expressions above do indeed coincide, we can use Chebyshev polynomials of the second kind and find that 
\begin{equation*}
 Z_n(q)= qU_{n-1}(\tfrac{q}{2}+1).
\end{equation*}
\end{proof}

We next move to the proof of \cref{prop:cycle_path_partition}, for which we will first need to expresses in the next lemma the probability of a boundary point in the path-graph being a root in terms of differences of the partition function. 

\begin{lemma}[Rooting events in path-graphs]\label{lem:path_root_measure}
Let $PG_n$ be the path-graph on $n$ vertices and $Z_n(q)$ its partition function.
Let $x \in V$ be a vertex with distance $d \in \N_0$ from the boundary and $b \in V$ a boundary vertex.
Let $\NNM^{(n)}$ denote the non-normalized measure on $PG_n$ and $R_q$ the set of roots of $\Phi_q$. Then
\begin{equation}\label{eq:splitting_rooting_events}
 \NNM^{(n)}(x \in R_q) = \tfrac{1}{q}\NNM^{(d+1)}(b \in R_q) \ \NNM^{(n-d)}(b \in R_q),
\end{equation}
with 
\begin{equation}\label{eq:rooting_recurrence}
\NNM^{(n)}(b \in R_q)=Z_{n}(q)-Z_{n-1}(q).
\end{equation}
For the non-normalized measure of the event that both boundary vertices $b$ and $b'$ are roots it holds that
 \begin{equation}\label{eq:path_double_boundary_roots}
  \NNM^{(n)}(b,b' \in R_q) = q Z_{n-1}(q).
 \end{equation}
\end{lemma}

\begin{proof}[Proof of \cref{lem:path_root_measure}]
\leavevmode \\
\fbox{\cref{eq:splitting_rooting_events}}
Let $L_n$ denote the graph Laplacian of the path-graph on $n$ vertices.
Inspection of the Laplacian and using the symmetry of the path-graph shows that 
\begin{equation*}
 \det[qI-L_n]_x = \det[qI-L_{d+1}]_b \det[qI-L_{n-d}]_b,
\end{equation*}
as removing a row and column from $qI-L_n$ results in a matrix comprised of two blocks.
Since the event that vertex $x$ is a root equals the event that none of the outgoing edges of $x$ are present,
it holds by \cref{lem:spatial_markov_property} that $\NNM^{(n)}(x \in R_q)=q \ \det[qI-L_n]_x$, 
from which \cref{eq:splitting_rooting_events} follows.

\noindent\fbox{\cref{eq:rooting_recurrence}} 
Since $\NNM^{(n)}(b \in R_q)=Z_{n}(q)-\NNM^{(n)}(b \notin R_q)$, \cref{eq:rooting_recurrence} follows directly from \cref{eq:path_rooting_extension}.

\noindent\fbox{\cref{eq:path_double_boundary_roots}}
By \cref{lem:graph_extension_single_edge} we have that
\begin{equation*}
 \NNM^{(n)}(b,b' \in R_q) = q\NNM^{(n-1)}(b \in R_q) + \NNM^{(n-1)}(b,b' \in R_q), \text{ and } \ 
 \NNM^{(n)}(b \in R_q, \ b' \notin R_q) = \NNM^{(n-1)}(b \in R_q).
\end{equation*}
Since $\NNM^{(n-1)}(b,b' \in R_q)=\NNM^{(n-1)}(b \in R_q)-\NNM^{(n-1)}(b \in R_q, \ b' \notin R_q)$, 
it follows from \cref{eq:path_laplacian_recurrence,eq:rooting_recurrence} that 
\begin{align*}
 \NNM^{(n)}(b,b' \in R_q) &= (q+1)\NNM^{(n-1)}(b \in R_q) - \NNM^{(n-2)}(b \in R_q) \\
 &= (q+1)Z_{n-1}(q) - (q+2)Z_{n-2}(q) + Z_{n-3}(q) 
 = qZ_{n-1}(q).
\end{align*}
\end{proof}

\begin{proof}[Proof of \cref{prop:cycle_path_partition}]
We will first prove \cref{eq:cycle_partition_path}.
Let $V$ denote the vertex set of $CG_n$ and let $x \in V$ be a vertex.
The partition function can be split into two terms
\begin{align}
 Z_{CG_n}(q) &= \NNM^{(CG_n)}(x \in R_q) + \NNM^{(CG_n)}(x \notin R_q).
\end{align}

Note that the induced subgraph $CG_n[V\setminus\cbrac{x}]$ obtained by removing vertex $x$, is a path-graph on $n-1$ vertices.
Let $y$ and $z$ denote the two vertices adjacent to $x$ in $CG_n$. So, these are the boundary vertices of $PG_{n-1}$. 
We will use \cref{lem:graph_extension_single_vertex}.
This gives us by \cref{eq:path_laplacian_recurrence,lem:path_root_measure} that 
\begin{align*}
 \NNM^{(CG_n)}(x \in R_q) &= 
 \sum_{F \in \mathcal{F}_{PG_{n-1}}} q \ \NNM^{(PG_{n-1})}(\Phi=F) \ (1+\tfrac{1}{q})^{\abs{R(F) \cap \cbrac{y,z}}} \\
 &=(q+2+\tfrac{1}{q})\NNM^{(PG_{n-1})}(y,z \in R_q) + 2(q+1)\NNM^{(PG_{n-1})}(y \in R_q, z \notin R_q) 
 + q \NNM^{(PG_{n-1})}(y,z \notin R_q) \\
 &=qZ_{PG_{n-1}}(q)+(2+\tfrac{1}{q})\NNM^{(PG_{n-1})}(y,z \in R_q) + 2\NNM^{(PG_{n-1})}(y \in R_q, z \notin R_q) \\
 &=(q+2)Z_{PG_{n-1}}(q) - 2 Z_{PG_{n-2}}(q) + \tfrac{1}{q}\NNM^{(PG_{n-1})}(y,z \in R_q) \\
 &=Z_{PG_{n}}(q) -  Z_{PG_{n-2}}(q) + \tfrac{1}{q}\NNM^{(PG_{n-1})}(y,z \in R_q) = Z_{PG_{n}}(q).
\end{align*}

Let $r_y(F)$ denote the root in the tree of forest $F$ that contains vertex $y$. 
Again using \cref{lem:graph_extension_single_vertex} and \cref{eq:path_laplacian_recurrence}, we obtain 
 \begin{align*}
  \NNM^{(CG_n)}&(x \notin R_q) = \sum_{F \in \mathcal{F}_{PG_{n-1}}} \NNM^{(PG_{n-1})}(\Phi=F) \ 
     2 (1+\tfrac{1}{q})^{\mathbf{1}\cbrac{z \in R(F), \ r_y(F) \neq z}} \\
  &= 2\NNM^{(PG_{n-1})}(z \notin R_q) + 2 (1+\tfrac{1}{q})\NNM^{(PG_{n-1})}(z \in R_q, r_y(F) \neq z) = (2+\tfrac{2}{q})Z_{PG_{n-1}}(q)-\tfrac{2}{q}Z_{PG_{n-2}}(q)-2 \\
  &= \tfrac{2}{q}((q+2)Z_{PG_{n-1}}(q)-Z_{PG_{n-2}}(q)-Z_{PG_{n-1}}(q))-2 = \tfrac{2}{q}(Z_{PG_{n}}(q)-Z_{PG_{n-1}}(q))-2.
 \end{align*}
This proves \cref{eq:cycle_partition_path}.

\Cref{eq:cycle_partition_combinatorial} follows from \cref{eq:cycle_partition_path} 
and the expression for the path-graph partition function given in \cref{eq:path_partition_combinatorial}, by repeated applications of Pascal's formula.
\end{proof}

\subsection{Asymptotic analysis of path-graphs}
\begin{proof}[Proof of \cref{pathLEP}]
\leavevmode 
\fbox{\emph{\cref{eq:path_correlation}}}
Let $\mathcal{F}_n$ denote the set of rooted forests of $PG_n$ and write
\begin{align*}
 \mathcal{F}_{n-d}^{k} &= \cbrac{F \in \mathcal{F}_{n-d} \c r(F) = k}; \\
 \mathcal{R}_{n-d}^{k}(x) &= \cbrac{F \in \mathcal{F}_{n-d} \c r(F) = k, \ x \in R(F)}; \\
 \mathcal{C}_{n}^{k}(x,y) &= \cbrac{F \in \mathcal{F}_{n} \c r(F) = k, \ x \leftrightarrow_F y}.
\end{align*}
It is sufficient to show that for all $k \in [n-d]$ it holds that
$|\mathcal{C}_{n}^{k}(x,y)| = |\mathcal{F}_{n-d}^{k}| + d|\mathcal{R}_{n-d}^{k}(x)|$.
The result then follows from \cref{lem:path_root_measure}.

We will construct a bijection between the set $\mathcal{C}_{n}^{k}(x,y)$ and the set 
$\mathcal{F}_{n-d}^{k} \cup (\mathcal{R}_{n-d}^{k}(x) \times [d])$.
Let $F \in \mathcal{C}_{n}^{k}(x,y)$ be given. Let $r \in [n]$ denote the vertex of $F$ 
that is the root in the component of $x$ and $y$. Let $B=[y]\setminus[x-1]$ denote the set of all vertices from $x$ to $y$
and let $F_B \in \mathcal{F}_{n-d}^{k}$ denote the $B$-vertex contraction of $F$.
Then we have that $F_B \in \mathcal{R}_{n-d}^{k}(x)$ if and only if $r \in [y]\setminus[x-1]$.
Define the function $f:\mathcal{C}_{n}^{k}(x,y) \to \mathcal{F}_{n-d}^{k} \cup (\mathcal{R}_{n-d}^{k}(x) \times [d])$ by
\begin{equation*}
 f(F) =
 \begin{cases}
  F_B & \text{ if } r \notin [y]\setminus[x] \\
  (F_B, r-x) & \text{ if } r \in [y]\setminus[x].
 \end{cases}
\end{equation*}
It is easily verified that this gives a bijection.

\leavevmode \\
\fbox{\emph{Lower bound}}
Let $\tilde{\P}_x$ denote the law of the discrete-time random walk $\tilde{X}$ on $PG_n$ starting on $x$, as defined in \cref{eq:discrete_time_rw}.
Since in this case we consider a path-graph, we have that $\tilde{\tau}_{q} \sim \mathrm{Geom}(\tfrac{q}{q+2})$.

We will analyze the expression in \cref{LEdec}.
Let $z$ denote a vertex halfway between $x$ and $y$. 
For notational simplicity we assume that $d$ is even, so that $z=x+\tfrac{d}{2}$.
The argument in the case where $d$ is odd is similar.
Note that the vertices $x$ and $y$ are disconnected in $\Phi_q$ 
if both the random walks starting at $x$ and the random walk starting at $y$ are killed before reaching vertex $z$.
So, we have that
\begin{align}
\P(x \nleftrightarrow_{\Phi_q} y) & \geq \tilde{\P}_x(\tilde{\tau}_{q} \leq \tau_{z})\tilde{\P}_y(\tilde{\tau}_{q} \leq \tau_{z}).
\end{align}

Let $\tau_{S}(k)$ denote the hitting time of $k \in \Z$ by $S$.
A coupling of $\tilde{X}$ and $S$ can be used to show that
\begin{align}
 \tau_{z} \stackrel{d}{=}\min\cbrac{\tau_{S}\paren*{\tfrac{d}{2}}, \tau_{S}\paren*{1-2x-\tfrac{d}{2}}},
\end{align}
where $\stackrel{d}{=}$ denotes equality in distribution.

By the reflection principle it holds for all $k,n \in \N$ that $\P(\tau_{S}(n) \leq k)=\P(S_k \notin [-n,n-1])$.
For $u \in \cbrac{x,y}$ it follows that
\begin{align*}
 \tilde{\P}_u(\tilde{\tau}_{q} \leq \tau_z) 
 &=  \sum_{k=1}^{\infty}\P(\tilde{\tau}_{q} = k)\tilde{\P}_u(\tau_{z} \geq k) 
\geq \sum_{k=1}^{m}\paren*{1-\tilde{\P}_u(\tau_{z} < k)}\P(\tilde{\tau}_{q} = k)\\
 &=\sum_{k=1}^{m}\paren*{1-\P(\tau_S(\tfrac{d}{2}) < k \text{ or } 
 \tau_S(1-2x-\tfrac{d}{2}) < k)}\P(\tilde{\tau}_{q} = k) \\
 &\geq\sum_{k=1}^{m}\paren*{1-2\P(\tau_S(\tfrac{d}{2}) < k)}\P(\tilde{\tau}_{q} = k) 
 =\sum_{k=1}^{m}\paren*{2\P(S_{k-1} \in [-\tfrac{d}{2},\tfrac{d}{2}-1])-1}
 \P(\tilde{\tau}_{q} = k)\\
 & \geq \sum_{k=1}^{m}\paren*{2\P(\abs{S_{k-1}} < \tfrac{d}{2})-1}\P(\tilde{\tau}_{q} = k)
 \geq \sum_{k=1}^{m}\paren*{2\P(\abs{S_{m}} < \tfrac{d}{2})-1}\P(\tilde{\tau}_{q} = k) \\
&=\paren*{2\P\paren*{\abs{S_{m}} < \tfrac{d}{2}}-1}\P(\tilde{\tau}_{q} \leq m) 
=\paren*{2\P\paren*{\abs{S_{m}} < \tfrac{d}{2}}-1}\paren*{1-\paren*{1-\tfrac{q}{2+q}}^{m}}
\end{align*}
Hence
\begin{equation*}
\min_{u\in\{x,y\}}\tilde{\P}_u(\tilde{\tau}_{q} \leq \tau_{z}) \geq \paren*{2\P\paren*{\abs{S_{m}} < \tfrac{d}{2}}-1}\paren*{1-\paren*{1-\tfrac{q}{2+q}}^{m}}.
\end{equation*}
If $\paren*{2\P\paren*{\abs{S_{m}} < \tfrac{d}{2}}-1}$ is non-negative, then we also have that
\begin{equation*}
\tilde{\P}_x(\tilde{\tau}_{q} \leq \tau_{z})\tilde{\P}_y(\tilde{\tau}_{q} \leq \tau_{z}) \geq 
\paren*{2\P\paren*{\abs{S_{m}} < \tfrac{d}{2}}-1}^2\paren*{1-\paren*{1-\tfrac{q}{2+q}}^{m}}^2.
\end{equation*}
Therefore, we have for all $m \in \N$ with $\P\paren*{\abs{S_{m}} < \frac{d}{2}} \geq \tfrac{1}{2}$ that 
\begin{align*}
 U_{q}^{(n)}(x,y) &\geq \paren*{2\P\paren*{\abs{S_{m}} < \tfrac{d}{2}}-1}^2\paren*{1-\paren*{1-\tfrac{q}{2+q}}^{m}}^2,
\end{align*}
which gives the desired lower bound.

\leavevmode \\
\fbox{\emph{Upper bound}}
We again analyze by means of Wilson's algorithm with the first random walk starting at $x$ and the second one starting at $y$.
Note that the trajectory of the loop-erasure of the first random walk will always contain its starting vertex $x$.
Thus if the second random walk hits $x$ before being killed, then $x$ and $y$ are connected in $\Phi_q$.
Therefore, we have that
\begin{align*}
 \P(x \leftrightarrow_{\Phi_{q}} y) \geq \tilde{\P}_y(\tau_{x} < \tilde{\tau}_{q}).
\end{align*}

Using a coupling argument we can show that 
\begin{align*}
 \tau_{x}&\stackrel{d}{=} \min\cbrac{\tau_{S}(-d), \tau_{S}(2n+d-2y+1)},
\end{align*}
where $\tau_x$ denotes the first hitting time vertex $x$ by the random walk $\tilde{X}$ starting at $y$.
So, in a manner similar to that used for the lower bound, we find for all $m \in \N$ that
\begin{align*}
 \tilde{\P}_{y}(\tau_{x} < \tilde{\tau}_{q} ) &= \sum_{k=1}^{\infty}\tilde{\P}_{y}(\tau_{x} < k) \P(\tilde{\tau}_{q} = k) 
 \geq \sum_{k=m}^{\infty}\tilde{\P}_{y}(\tau_{x} \leq k) \P(\tilde{\tau}_{q} = k + 1) \\
 &= \sum_{k=m}^{\infty}\P(\tau_{S}(-d) \leq k \text{ or } \tau_{S}\paren*{2n+d-2y+1} \leq k) \P(\tilde{\tau}_{q} = k + 1) \\
 &\geq \sum_{k=m}^{\infty}\P(\tau_{S}(-d) \leq k) \P(\tilde{\tau}_{q} = k + 1) \geq \sum_{k=m}^{\infty}\P(\tau_{S}(d) \leq a) \P(\tilde{\tau}_{q} = k + 1) \\
 &= \P(\tau_{S}(d) \leq m) \P(\tilde{\tau}_{q} > m) 
 = \P(S_{m} \notin [-d,d-1]) \P(\tilde{\tau}_{q} > m) \\
 &\geq \P(\abs{S_{m}} > d) \P(\tilde{\tau}_{q} > m) = \P\paren*{\abs{S_{m}} > d} \paren*{1-\tfrac{q}{2+q}}^{m}.
\end{align*}
It follows that 
\begin{align*}
 U_{q}^{(n)}(x,y) &= 1- \P(x \leftrightarrow_{\Phi_q} y) 
 \leq 1- \P\paren*{\abs{S_{m}} > d} \paren*{1-\tfrac{q}{2+q}}^{m}.
\end{align*}
\end{proof}

\begin{proof}[Proof of \cref{thm:path_correct_scaling}]
\leavevmode \\
\fbox{\emph{$q_n = \smallO(\tfrac{1}{d_n^2})$}}
Set $m_n=\ceil{\frac{d_n}{\sqrt{q_n}}}$, i.e. $m_n$ is the smallest integer that is not smaller than $\frac{d_n}{\sqrt{q_n}}$.
We have that $m_n=\smallOmega(d_n^2)$.
In particular this means that $m_n \to \infty$ as $n \to \infty$.
So, $\frac{S_{m_n}}{\sqrt{m_n}}$ converges in distribution to a standard normal random variable.
Since $\tfrac{d_n}{\sqrt{m_n}} \to 0$, it follows that $\P\paren*{\tfrac{\abs{S_{m_n}}}{\sqrt{m_n}} > \tfrac{d_n}{\sqrt{m_n}}} \to 1$.
We also have that $m_n = \smallO(\tfrac{1}{q_n})$, which gives us that $\paren*{1-\frac{q_n}{2+q_n}}^{m_n} \to 1$.
Therefore, the upper bound from \cref{pathLEP} gives us that
\begin{align*}
 U_{q_n}^{(n)}(x_n,y_n) &\leq 
 1-\P\paren*{\tfrac{\abs{S_{m_n}}}{\sqrt{m_n}} > \tfrac{d_n}{\sqrt{m_n}}}\paren*{1-\frac{q_n}{2+q_n}}^{m_n} 
=\smallO(1).
\end{align*}
\leavevmode \\
\fbox{\emph{$q_n = \smallOmega(\tfrac{1}{d_n^2})$}}
Again set $m_n=\ceil{\frac{d_n}{\sqrt{q_n}}}$.
It holds that $m_n = \smallOmega(\tfrac{1}{q_n})$, so that $\paren*{1-\frac{q_n}{2+q_n}}^{m_n} \to 0$.
Furthermore, we have that $m_n=\smallO(d_n^2)$. This means that $\tfrac{d_n}{2\sqrt{m_n}} \to \infty$ and thus that 
$\P\paren*{\tfrac{\abs{S_{m_n}}}{\sqrt{m_n}} < \tfrac{d_n}{2\sqrt{m_n}}} \to 1$. 
For large enough $n$, this gives us that $\P\paren*{\tfrac{\abs{S_{m_n}}}{\sqrt{m_n}} < \tfrac{d_n}{2\sqrt{m_n}}} \geq \tfrac{1}{2}$,
which means that we can apply the lower bound from \cref{pathLEP}.
This gives us that
\begin{align*}
 U_{q_n}^{(n)}(x_n,y_n) &\geq 
 \paren*{1-\paren*{1-\frac{q_n}{2+q_n}}^{m_n}}^2\paren*{2\P\paren*{\frac{\abs{S_{m_n}}}{\sqrt{m_n}} < \frac{d_n}{2\sqrt{m_n}}}-1}^2 
  = 1- \smallO(1).
\end{align*}
\leavevmode \\
\fbox{\emph{$q_n = \tfrac{c}{d_n^2}+\smallO(\tfrac{1}{d_n^2})$}}
Now set $m_n=\ceil{\frac{d_n}{4\sqrt{c q_n}}}$.
We will distinguish between the case where $d_n$ diverges and the case where $d_n$ is bounded.

First assume that $d_n=\smallOmega(1)$.
Then we find that $m_n \sim \frac{1}{4q_n}$.
It follows that there exists a $\varepsilon > 0$ small enough that 
$\varepsilon<\paren*{1-\frac{q_n}{2+q_n}}^{m_n}<1-\varepsilon$ for all $n \in \N$.
We also have that $m_n \sim \tfrac{1}{4}d_n^2$. This gives us that
$\frac{S_{m_n}}{\sqrt{m_n}}$ converges in distribution to a standard normal random variable $Z$
and that $\frac{d_n}{2\sqrt{m_n}}\to 1$. Since $0.6<\P(\abs{Z}<1) < 0.7$, we can apply the lower bound from \cref{pathLEP}.
Using both bounds from \cref{pathLEP}, we conclude the non-degeneracy:
\begin{equation*}
 \lim_{n \to \infty}U_{q_n}^{(n)}(x_n,y_n) \in (0,1).
\end{equation*}

Now instead assume that $d_n$ is bounded, i.e. there exists an $M \in \N$ with $M \geq d_n$ for all $n \in \N$.
Then the lower bound from \cref{pathLEP} can not necessarily be applied.
However, we can lower bound the probability that $x_n$ and $y_n$ are disconnected by the probability that 
the discrete-time random walks on $PG_n$ starting at $x$ and $y$ are both killed at time $1$, while still at their starting points.
This probability equals $\P(\tilde{\tau}_{q}=1)^2=\frac{q_n^2}{(2+q_n)^2}$.

The probability that $x_n$ and $y_n$ are connected can be lower bounded by the probability that the discrete-time random walk on $PG_n$ 
starting at $x$ jumps $M$ times in the direction of $y$ and then is then killed at time $M+1$.
This probability equals $\paren*{\tfrac{1}{2+q_n}}^M \tfrac{q_n}{2+q_n}$.
So, we have for all $n \in \N$ that
\begin{align*}
 \frac{q_n^2}{(2+q_n)^2} \leq U_{q_n}^{(n)}(x_n,y_n) \leq 1-\paren*{\tfrac{1}{2+q_n}}^M \tfrac{q_n}{2+q_n}.
\end{align*}
Since $q_n \sim \tfrac{c}{d_n^2}$ and $d_n$ is bounded, we have that $q_n$ is bounded away from $0$ and away from infinity. 
Hence, the 2-point correlation is also non-degenerate in this case.
\end{proof}

\begin{proof}[Proof of \cref{eq:path_asymptotics_root_distance}]
We start the proof with three technical limits.
Let $\alpha \in \R$ be a constant. 
We claim that
 \begin{equation}\label{eq:path_root_asymptotics_part1}
  \lim_{n \to \infty} \paren*{\frac{2}{q_n + 2 + \sqrt{q_n^2 + 4 q_n}}}^{\alpha\sqrt{n}+\smallO(\sqrt{n})}= e^{-\tfrac{\alpha}{2\delta}},
 \end{equation} 
 \begin{equation}\label{eq:path_root_asymptotics_part2}
  \lim_{n \to \infty} \paren*{\frac{q_n + 2 - \sqrt{q_n^2 + 4 q_n}}{q_n + 2 + \sqrt{q_n^2 + 4 q_n}}}^{\alpha\sqrt{n}+\smallO(\sqrt{n})}= e^{-\frac{\alpha}{\delta}}
 \end{equation}
 and
 \begin{equation}\label{eq:path_root_asymptotics_part3}
  \lim_{n \to \infty} \paren*{\frac{q_n + 2 - \sqrt{q_n^2 + 4 q_n}}{q_n + 2 + \sqrt{q_n^2 + 4 q_n}}}^{\smallOmega(\sqrt{n})}=0.
 \end{equation}
Each of the three identities will be proven separately. 
\leavevmode \\
\fbox{\cref{eq:path_root_asymptotics_part1}}
 Since $\tfrac{\sqrt{1+4d_n^2}}{2d_n^2}- \tfrac{1}{2\delta \sqrt{n}}=\smallO(\tfrac{1}{\sqrt{n}})$, 
 we have that
 \begin{align*}
  \hspace{-3 em}\paren*{\frac{q_n+2+\sqrt{q_n^2+4q_n}}{2}}^{\alpha\sqrt{n}+\smallO(\sqrt{n})} 
  &= \paren*{1+\frac{\sqrt{1+4d_n^2}}{2d_n^2}+\frac{1}{2d_n^2}}^{\alpha\sqrt{n}+\smallO(\sqrt{n})} \\
  &= \paren*{1+\frac{1}{2\delta \sqrt{n}}+\smallO(\tfrac{1}{\sqrt{n}})}^{\alpha \sqrt{n}+\smallO(\sqrt{n})} 
  = e^{\frac{\alpha}{2\delta}}+\smallO(1).   
 \end{align*}

\leavevmode \\
\fbox{\cref{eq:path_root_asymptotics_part2}}
Note that $\sqrt{q_n^2+4q_n}= \tfrac{1}{\delta \sqrt{n}}(1+\smallO(1))$. 
Hence,
 \begin{align*}
\hspace{-3 em}\paren*{\frac{q_n+2-\sqrt{q_n^2+4q_n}}{q_n+2+\sqrt{q_n^2+4q_n}}}^{\alpha\sqrt{n}+\smallO(\sqrt{n})}  
&=\paren*{1-\frac{2\sqrt{q_n^2+4q_n}}{q_n+2+\sqrt{q_n^2+4q_n}}}^{\alpha\sqrt{n}+\smallO(\sqrt{n})}  
=\paren*{1-\frac{\sqrt{q_n^2+4q_n}}{1+\smallO(1)}}^{\alpha\sqrt{n}+\smallO(\sqrt{n})} \\
&=\paren*{1-\frac{1}{\delta \sqrt{n}}(1+\smallO(1))}^{\alpha \sqrt{n}+\smallO(\sqrt{n})}  =e^{-\frac{\alpha}{\delta}}+\smallO(1).
 \end{align*}
 
\leavevmode \\
 \fbox{\cref{eq:path_root_asymptotics_part3}}
Similar to \cref{eq:path_root_asymptotics_part2} it holds that
\begin{align*}
\paren*{\frac{q_n+2-\sqrt{q_n^2+4q_n}}{q_n+2+\sqrt{q_n^2+4q_n}}}^{\smallOmega(\sqrt{n})}
=\paren*{1-\frac{1}{\delta \sqrt{n}}(1+\smallO(1))}^{\smallOmega(\sqrt{n})} \to 0.
 \end{align*}
 This concludes the proof of the claim.

Now that we have established these identities we continue with the proof.
For brevity write $Z_n(q)=Z_{PG_n}(q)$.
Using the expression for the partition function given in \cref{eq:path_partition_recurrence} we have for each $m \in \N$ that
 \begin{equation}\label{eq:path_partition_shorthand}
 Z_m(q_n)=\frac{1}{\sqrt{1+4d_n^2}}\paren*{1-\paren*{\frac{q_n+2-\sqrt{q_n^2+4q_n}}{q_n+2+\sqrt{q_n^2+4q_n}}}^{m}}
 \paren*{\frac{q_n+2+\sqrt{q_n^2+4q_n}}{2}}^{m}.
\end{equation}
By \cref{eq:path_correlation} the 2-point correlation is given by
\begin{equation}\label{eq:path_correlation_square_root_proof}
\correlation[PG_n]{q_n}{x_n}{y_n}=1-\frac{Z_{n-d_n}(q_n)}{Z_{n}(q_n)}-\frac{d_n \left[Z_{x_n}(q_n)-Z_{x_n-1}(q_n)\right]
\left[Z_{n-y_n+1}(q_n)-Z_{n-y_n}(q_n)\right]}{q_nZ_{n}(q_n)}. 
\end{equation}
The result follows by plugging in \cref{eq:path_partition_shorthand} into \cref{eq:path_correlation_square_root_proof}
and repeatedly applying the limits in \cref{eq:path_root_asymptotics_part1,eq:path_root_asymptotics_part2,eq:path_root_asymptotics_part3}.
\end{proof}

\section{Asymptotic detection of modular structures}

\subsection{Star graphs: homogeneous case and with implanted communities}

\begin{proof}[Proof of \cref{Star}]
Let us start by providing an expression for the partition function of a regular tree with homogeneous weights.
Let $w \in (0,\infty)$, $h \in \N$, $k \in \N$, and let $L$ be the graph Laplacian of the $k$-regular tree with height $h$ and uniform weight $w$. Define $(\alpha_n)_{n \in \N_0}$ such that $\alpha_0 = q+w$ and $\alpha_{n+1} = q+(k+1)w - \frac{kw^2}{\alpha_n}$ for $n \in \N$. Then the characteristic polynomial of $L$ is given by
\begin{equation} \label{CP}
  \det[qI-L] = \left( \prod_{i=0}^{h-1} \alpha_i^{k^{h-i}} \right) \left( q + kw - \frac{kw^2}{\alpha_{h-1}} \right)
\end{equation}
In fact, observe that in the matrix $[qI-L]$ there is a $k^h\times k^h$ diagonal matrix with entries $q + w$ since the leaves are not connected with each other. Call this right lower diagonal matrix $D$ and call the corresponding left upper matrix $A$, 
right upper matrix $B$ and left lower matrix $C$. By Schur's determinant identity, we get $\det[qI-L] = \det[D]\det[A - BD^{-1}C]$. Here, $\det[D] = (q+w)^{k^h}$ since $D = (q+w)I$. This also gives  us $D^{-1} = \frac{1}{q + w}I$. 
Thus, $BD^{-1}C = \frac{1}{q+w}BC$ is a diagonal matrix with lower entries $\frac{kw^2}{q + w}$, on the places of the parents of the leaves, and upper entries $0$, on the places of the nodes that are not parents of the leaves (if there are any). 
If $h = 1$ we see that $A - BD^{-1}C = q + kw - \frac{kw^2}{q + w}$ and we are done.
If $h > 1$ we see that $A - BD^{-1}C$ is again a matrix with a right lower diagonal matrix. 
This time, the entries of the diagonal matrix are $q + (k+1)w - \frac{kw^2}{q + w}$. 
By iteration of Schur's determinant identity we get the formula in \eqref{CP}.  

We'll continue by checking the validity of the expressions in \eqref{eq:star_correlation_center} and \eqref{eq:star_correlation_leaves}.
By applying \eqref{CP} to the homogeneous star graph $SG_n$ we obtain that its partition function is given by
\begin{equation}\label{eq:star_partition}
 Z_{SG_n}(q)=q(q+w)^{n-2}(q+nw)
\end{equation}
Since $d(c,x) =1$, by \eqref{eq:close_factor} we have that
\begin{align*}
U_q(c,x) &= \frac{q \ Z_{SG_{n-1}}(q)}{Z_{SG_{n}}(q)} =\frac{q(q+(n-1)w)}{(q+w)(q+nw)}.
\end{align*}
Similarly, since $d(x,y) = 2$, by \cref{prop:inex} we have that
\begin{align*}
  U_q(x,y) &= \frac{2q \ Z_{SG_{n-1}}(q) \ - \ q^2 \ Z_{SG_{n-2}}(q)}{Z_{SG_{n}}(q)} = \frac{q(q^2 + (n+2)w q + 2(n-1)w^2)}{(q+w)^2(q+nw)},
\end{align*}
which finishes the proof of \eqref{eq:star_correlation_center} and \eqref{eq:star_correlation_leaves}. The asymptotics in \eqref{As1} and \eqref{As2} follow immediately.
\end{proof}

\begin{proof}[Proof of theorem \cref{CommStar}]
The Laplacian $L$ of the community star graph $CSG_{n,k}$ is given by
\[
	L = \left[ 
	\begin{array}{cccccccc}
		-k - (n-k-1)w & 1 & 1 & 1 & \cdots & w & w & w \\
		1 & -1 &  &  &  &  &  &  \\
		1 &  & -1 &  &  &  &  &  \\
		1 &  &  & -1 &  &  &  &  \\
		\vdots &&&& \ddots & && \\
		w &  &  &  &  & -w &  &  \\
		w &  &  &  &  &  & -w &  \\
		w &  &  &  &  &  &  & -w
	\end{array}\right]
\]
where the empty places are to be filled with zeros. The characteristic polynomial of this graph Laplacian is thus
\begin{equation} \label{eq:community_star_partition}
	\det[qI-L] = q(q+w)^{n-k-2}(q+1)^{k-1}[q^2 + ((n-k)w + k + 1)q + nq]
\end{equation}
which can be found by applying Schur's determinant identity as we did in the proof of \cref{Star}.
Hence, the eigenvalues of the graph Laplacian are:
\begin{equation*}
  \lambda_i = \left\{
  \begin{array}{ll}
    0 & \text{ if } \  i = 1 \\
    -w & \text{ if } \  i = 2,\ldots, n-k-1 \\
    -1 & \text{ if } \  i = n-k,\ldots,n-2\\
    \frac{1}{2}\mu + \frac{1}{2}\delta & \text{ if } \  i = n-1 \\
    \frac{1}{2}\mu - \frac{1}{2}\delta & \text{ if } \  i = n \\
  \end{array}\right.
\end{equation*}
where 
\begin{align*}
	\mu &= 
	(n-k)w + k + 1 \\
	\delta &=
	\sqrt{((n-1)^2 - 2nk + k^2 + 2n-1)w^2 + k^2 + 2((n-1)k-k^2-n)w + 2k + 1}.
\end{align*}
Denote the sets of vertices that are connected to the center vertex $c$ with a weight $1$ and $w$ by $V_1$ and $V_w$, respectivley.
Combining \cref{prop:inex} with \cref{eq:community_star_partition} leads to:
\begin{equation*}
U_q(c,x) = \left\{
  \begin{array}{ll}
    \frac{q(q^2 + ((n-k)w + k)q + (n-1)w)}{(q+1)(q^2 + ((n-k)w + k + 1)q + nw)} & x \in V_1 \\
    & \\
    \frac{q(q^2 + ((n-k-1)w + k + 1)q + (n-1)w)}{(q+w)(q^2 + ((n-k)w + k + 1)q + nw)} & x \in V_{w}
\end{array}\right.
\end{equation*}
and 
\begin{equation*}
U_q(x,y) =
  \left\{ 
  \begin{array}{ll}
    \frac{q(q^3 + ((n-k)w + (k+3))q^2 + ((3n - 2k)w + 2k)q + 2(n-1)w)}{(q+1)^2(q^2 + ((n-k)w + (k+1))q + nw)} & x, y \in V_1 \\
    & \\
    \frac{q(q^3 + ((n-k+1)w + (k+2))q^2 + ((n-k)w^2 + (2n-1)w + (k+1))q + (n-1)w(1+w))}{(q+1)(q+w)(q^2 + ((n-k)w + (k+1))q + nw)} & x \in V_1, y \in V_{w} \\
    & \\
    \frac{q(q^3 + ((n-k+2)w + (k+1))q^2 + ((n - 2k+2)w + 2(n-k-1)w^2)q + 2(n-1)w^2)}{(q+w)^2(q^2 + ((n-k)w + (k+1))q + nw)} 
    & x, y \in V_{w}
\end{array}\right.
\end{equation*}
From these explicit formulas, letting $q$ and $w$ be as in the statement, the limits in \cref{CommStar} follow.
\end{proof}

\subsection{Playing with degrees and hierarchical weights on trees}

\begin{proof}[Proof of \cref{thm:correlation_asymptotics_bounded_vertices}]
Note that $\P(e \in \Phi_q) \leq \tfrac{w(e)}{q+w(e)}$, since by \cref{lem:edge_probability} it holds that
\begin{equation*}
  1 \geq \P(e \in \Phi_q)+\P(x \in R_q, \ x \nleftrightarrow_q y)
  =\P(e \in \Phi_q)(1+\tfrac{q}{w(e)}).
\end{equation*}
Hence, if $q_k = \smallOmega(w_k(x,y))$, then we have that $\correlation{q_k}{x}{y} \to 1$.

Assume that $q_k = \smallO(w_k(x,y))$. 
Let $\tilde{\P}^{(k)}_x$ denote the law of the discrete-time random walk $\tilde{X}$ on $G_k$ starting at vertex $x \in V_k$ and let $\tilde{\tau}_{q}$ be a geometric killing time, as defined in \cref{eq:discrete_time_rw}. Let $\tau_x$ denote the first hitting time of $x$ by $\tilde{X}$. Let $m$ denote the number of vertices on the $x$-side of edge $(x,y)$ in $G$.
We will show by induction on $m$ that 
\begin{equation*}
 \tilde{\P}_x^{(k)}(\tau_y < \tilde{\tau}_{q}) = 1 - \bigTheta\paren*{\tfrac{q_k}{w_k(x,y)}}.
\end{equation*}

If $m=1$, then $x$ is a leaf in $G$.
It follows that
\begin{align*}
 \tilde{\P}_x^{(k)}(\tau_y < \tilde{\tau}_{q}) &= \frac{w_k(x,y)}{q_k+w_k(x,y)} \sim 1 - \tfrac{q_k}{w_k(x,y)}.
\end{align*}

Assume that $m \geq 2$. 
Let $N_x$ denote the set of neighbors of $x$ in $G$.
Since the limit $\lim_{k \to \infty}\tfrac{w_k(e)}{q_k}$ exists for all edges incident to $x$, 
we can partition $N_x\setminus\cbrac{y}$ into two parts:
the first part $N^{\leq}_x=\cbrac{v \in N_x \setminus \cbrac{y} \c w_k(x,v) = \bigO(q_k)}$
consists of all neighbors for which the weight of the edge between $x$ and 
that neighbor has no larger order than $q_k$;
the second part $N^{>}_x=\cbrac{v \in N_x \setminus \cbrac{y} \c w_k(x,v) = \smallOmega(q_k)}$
consists of the remaining neighbors.

Then for each $v \in N^{>}_x$ we have that $q_k=\smallO(w_k(x,v))$.
For each such $v$ it follows by the induction hypothesis that 
$\P_{v}^{(k)}(\tau_x < \tilde{\tau}_{q}) = 1 - \bigTheta\paren*{\tfrac{q_k}{w_k(x,v)}}$.
It follows that
\begin{align*}
 \tilde{\P}_x^{(k)}(\tau_y < \tilde{\tau}_{q}) &= \frac{w_k(x,y)}{q_k + w_k(x,y)
 +\sum_{v \in N_x \setminus\cbrac{y}} w_k(x,v)(1-\tilde{\P}^{(k)}_v(\tau_{x} < \tilde{\tau}_{q}))} \\
 &= \frac{w_k(x,y)}{w_k(x,y) + \bigTheta(q_k)+\sum_{v \in N^{\leq}_x} w_k(x_n,v)
 (1-\tilde{\P}^{(k)}_v(\tau_{x} < \tilde{\tau}_{q}))} \\
 &= \frac{w_k(x,y)}{w_k(x,y) + \bigTheta(q_k)} = 1-\bigTheta\paren*{\frac{q_k}{w_k(x,y)}}.
\end{align*}

Thus we have that $\tilde{\P}_x^{(k)}(\tau_y < \tilde{\tau}_{q}) = 1-\smallO(1)$,
from which it follows that $\correlation[G_k]{q_k}{x}{y} \to 0$.
\end{proof}

\begin{lemma}[Parent hitting asymptotics with small $q$ in hierarchical trees of bounded height]
\label{lem:parent_hitting_asymptotics}
For each $n \in \N$ let $G_n=(V_n,E_n,w_n)$ be a hierarchical tree of height $H=H_n$, see \cref{fig:regular_hierarchical_tree}. 
Denote the weight of an edge at height $i \in [H]$ in $G_n$ by $w_{i}^{(n)}$ and recall that $w_{1}^{(n)}\leq\ldots\leq w_{H}^{(n)}$.

For each $n \in \N$ let $y_n$ be a vertex in $G_n$ at height $h=h_{n}$ such that $H_n-h_n$ is constant in $n$.
Let $x_n$ denote the parent of $y_n$.
For each vertex $v$ in $G_n$ let $\ell_n(v)$ denote the number of vertices in $G_n$ that have $v$ in their ancestry.
Let $(q_n)_{n \in \N}$ be sequence of rooting parameters such that $q_n=\smallO\paren*{\tfrac{w_{h}^{(n)}}{\ell_n(y)}}$.

For each $n \in \N$ let $\tilde{\P}^{(n)}_x$ denote the law of the discrete-time random walk $\tilde{X}$ on $G_n$ starting at vertex 
$x \in V_n$ and let $\tilde{\tau}_{q}$ be a geometric killing time, as defined in \cref{eq:discrete_time_rw}.
Let $\tau_x$ denote the first hitting time of $x$ by $\tilde{X}$.
Then as $n \to \infty$ it holds that
\begin{equation*}
 \tilde{\P}_{y}^{(n)}(\tau_{x} < \tilde{\tau}_{q}) \sim 1-\frac{q_n\ell_n(y)}{w_{h}^{(n)}}.
\end{equation*}
\end{lemma}
 \begin{proof}
 Write $k=H_n-h_n$, which is independent of $n$. We proceed by induction on $k$.

For $k=0$ we have that all vertices $y_n$ are leaves. We then have that $\ell_n(y)=1$, so that $q_n=\smallO(w_{H}^{(n)})$.
It follows that
\begin{align*}
 \P_{y}^{(n)}(\tau_{x} < \tilde{\tau}_{q}) &= \frac{w_{H}^{(n)}}{w_{H}^{(n)}+q_n} \sim 1- \tfrac{q_n}{w^{(n)}_{H}}.
\end{align*}

Now assume that $k > 0$. 
Let $C^{(n)}_y \subseteq V_n$ denote the set of child vertices of $y_n$ in $G_n$.
Note that since $k>0$, we have for all $n$ that $C^{(n)}_y$ is non-empty.
For each $n \in \N$ let $z_n$ be a child of $y_n$. Note that $\tfrac{w_{h}^{(n)}}{\ell_n(y)} \leq \tfrac{w_{h+1}^{(n)}}{\ell_n(z)}$.
This means that $q_n=\smallO\paren*{\tfrac{w_{h+1}^{(n)}}{\ell_n(z)}}$.
Thus by the induction hypothesis we then have that
\begin{equation*}
 \P_{z}^{(n)}(\tau_{y} < \tilde{\tau}_{q}) \sim 1-\frac{q_n\ell_n(z)}{w_{h+1}^{(n)}}. 
\end{equation*}
Since this holds for all possible choices of sequences of children of $y_n$, 
\cref{lem:sum_convergence} stated at the end of this section gives us that
\begin{equation}\label{eq:sum_convergence}
 \sum_{z \in C_y^{(n)}}\P_{z}^{(n)}(\tau_{y} < \tilde{\tau}_{q}) 
 \sim \sum_{z \in C_y^{(n)}} 1-\frac{q_n\ell_n(z)}{w_{h+1}^{(n)}}. 
\end{equation}

Note that for all $n$ it holds that
\begin{align*}
\tilde{\P}_{y}^{(n)}(\tau_{x} < \tilde{\tau}_{q}) &= \tilde{\P}_{y}^{(n)}(X_1 = x) 
+ \tilde{\P}_{y}^{(n)}(X_1 = y)\tilde{\P}_{y}^{(n)}(\tau_{x} < \tilde{\tau}_{q}) 
 + \sum_{z \in C_y^{(n)}}\tilde{\P}_{y}^{(n)}(X_1 = z)\tilde{\P}_{z}^{(n)}(\tau_{y} < \tilde{\tau}_{q})\tilde{\P}_{y}^{(n)}(\tau_{x} < \tilde{\tau}_{q}). 
\end{align*}
Solving this equation gives us that 
\begin{align}\notag
\tilde{\P}_{y}^{(n)}(\tau_{x} < \tilde{\tau}_{q}) &= 
\frac{\tilde{\P}_{y}^{(n)}(\tilde{X}_1 = x)}{1-\tilde{\P}_{y}^{(n)}(\tilde{X}_1 = y)
-\sum_{z \in C_y^{(n)}}\tilde{\P}_{y}^{(n)}(\tilde{X}_1 = z)\tilde{\P}_{z}^{(n)}(\tau_{y} < \tilde{\tau}_{q})} \\
&= \frac{w_{h}^{(n)}}{q_n + w_{h}^{(n)} + w_{h+1}^{(n)}\sum_{z \in C_y^{(n)}}\paren*{1-\tilde{\P}_{z}^{(n)}(\tau_{y} < \tilde{\tau}_{q})}}. \label{eq:parent_hitting_recursive}
\end{align}
Since $q_n=\smallO\paren*{\tfrac{w_{h}^{(n)}}{\ell_n(y)}}$, we then have that
\begin{align*}
 \tilde{\P}_{y}^{(n)}(\tau_{x} < \tilde{\tau}_{q}) &= 
 \frac{w^{(n)}_{h}}{q_n + w^{(n)}_{h} + w^{(n)}_{h+1}\sum_{z \in C^{(n)}_y}\paren*{1-\tilde{\P}^{(n)}_{z}(\tau_{y} < \tilde{\tau}_{q})}}  \sim \frac{w^{(n)}_{h}}{q_n + w^{(n)}_{h} + w^{(n)}_{h+1}\sum_{z \in C^{(n)}_y}\frac{q_n\ell_n(z)}{w_{h+1}^{(n)}}} \\
 &= \frac{w^{(n)}_{h}}{q_n + w^{(n)}_{h} + q_n\sum_{z \in C^{(n)}_y}\ell_n(z)} = \frac{w^{(n)}_{h}}{w^{(n)}_{h} + q_n\ell_n(y)}  \sim 1-\frac{q_n\ell_n(y)}{w_{h}^{(n)}}.
\end{align*}
\end{proof}

\begin{proof}[Proof of \cref{lem:hitting_time_expression}]
We will reason using the representation in~\eqref{LEdec} coming from Wilson's sampling construction.
We note in particular that in order for the directed edge $(x,y)$ to be present in $\Phi_q$ is equivalent to require that the loop-erased trajectory in~\eqref{LEdec} includes $y$, which can be expressed in terms of hitting times of the random walk as 
\begin{align*}
 \P((x,y) \in \Phi_q) &= \P_x(\tau_y < \tau_{q})
 \sum_{k=0}^{\infty}\paren*{\P_y(\tau_x < \tau_{q})\P_x(\tau_y < \tau_{q})}^k\P_y(\tau_{q} < \tau_x) \\
 &=\frac{\P_x(\tau_y < \tau_{q})\paren*{1-\P_y(\tau_x < \tau_{q})}}{1-\P_x(\tau_y < \tau_{q})\P_y(\tau_x < \tau_{q})}, 
\end{align*}
where the above sum runs over the number of times that the random walk reaches $y$. We notice in particular that the above step is equivalent to the using the forest transfer-current kernel in~\cite{avena2018transfercurrent}.

For the reversed edge $(y,x)$, we can write
\begin{align*}
 \P((y,x) \in \Phi_q) &= \paren*{1-\P((x,y) \in \Phi_q)}\P_y(\tau_x < \tau_{q}),
\end{align*}
where these two factors correspond to~\eqref{LEdec}. 
Therefore, it follows that
\begin{align*}
 \correlation[G]{q}{x}{y} &= 1-\P(x \leftrightarrow_{\Phi_q} y) = 1-\P((x,y) \in \Phi_q)-\P((y,x) \in \Phi_q) \\
 &=1-\tfrac{\P_x(\tau_y < \tau_{q})\paren*{1-\P_y(\tau_x < \tau_{q})}}{1-\P_x(\tau_y < \tau_{q})\P_y(\tau_x < \tau_{q})}
 -\paren*{1-\tfrac{\P_x(\tau_y < \tau_{q})\paren*{1-\P_y(\tau_x < \tau_{q})}}{1-\P_x(\tau_y < \tau_{q})\P_y(\tau_x < \tau_{q})}}\P_y(\tau_x < \tau_{q}) \\ 
 &=\frac{1-\P_x(\tau_y < \tau_{q})-\P_y(\tau_x < \tau_{q})+\P_x(\tau_y < \tau_{q})\P_y(\tau_x < \tau_{q})}{1-\P_x(\tau_y < \tau_{q})\P_y(\tau_x < \tau_{q})}.
\end{align*}
\end{proof}

\begin{proof}[Proof of \cref{thm:correlation_asymptotics_regular_hierarchical}]
 If $d_n$ is bounded, then the result follows from \cref{thm:correlation_asymptotics_bounded_vertices}.
 Hence, we can assume that $d_n\to \infty$ as $n \to \infty$.
 Since $G_n$ is a regular tree, the number of vertices with $y_n$ in their ancestry is given by $\ell_n(y)=\sum_{i=0}^{k}d_n^i$. This means that $\ell_n(y) \sim d_n^k$ as as $n \to \infty$. Hence, the case $q_n = \smallO\paren*{\tfrac{w_{n}(e_n)}{d_n^k}}$ follows directly from \cref{lem:parent_hitting_asymptotics,lem:hitting_time_expression}.
 
Assume that $q_k = \smallOmega\paren*{\tfrac{w_{k}(e_k)}{d_k^m}}$. 
By \cref{thm:monotone_correlations_symmetric} we can assume without loss of generality that also $q_n=\smallO\paren*{\tfrac{w_n(e_n)}{d_n^{k-1}}}$.
 
For ech $n \in \N$ let $\tilde{\P}^{(n)}_x$ denote the law of the discrete-time random walk $\tilde{X}$ on $G_n$ starting at vertex 
$x \in V_n$ and let $\tilde{\tau}_{q}$ be a geometric killing time, as defined in \cref{eq:discrete_time_rw}.
Let $\tau_x$ denote the first hitting time of $x$ by $\tilde{X}$.
 By \cref{lem:hitting_time_expression} it is sufficient to show that both $\tilde{\P}^{(n)}_y(\tau_{x} < \tilde{\tau}_{q}) \to 0$ and $\tilde{\P}^{(n)}_x(\tau_{y} < \tilde{\tau}_{q}) \to 0$ as $n \to \infty$.
 
 First we consider $\tilde{\P}^{(n)}_y(\tau_{x} < \tilde{\tau}_{q})$.
 Let $z_k$ be a child vertex of $y_k$.
 Then by \cref{lem:parent_hitting_asymptotics} we have that
 \begin{equation*}
  \tilde{\P}^{(n)}_z(\tau_{y} < \tilde{\tau}_{q}) \sim 1-\frac{q_n\sum_{i=0}^{m-1}d_n^{i}}{w_n(y_n,z_n)}
 \end{equation*}
 So, by using that $G_n$ is a regular tree, we have analogous to \cref{eq:parent_hitting_recursive} that
 \begin{align*}
 \tilde{\P}^{(n)}_y(\tau_{x} < \tilde{\tau}_{q}) &= 
 \frac{w_{n}(e_n)}{q_n + w_{n}(e_n) + d_n\paren*{1-\tilde{\P}^{(n)}_z(\tau_{y} < \tilde{\tau}_{q})} w_{n}(y_n, z_n) }\\
 &\sim \frac{w_{n}(e_n)}{q_n+w_{n}(e_n) + d_nq_n\sum_{i=0}^{m-1}d_n^{i}}= \frac{w_{n}(e_n)}{q_n+w_{n}(e_n) + \smallOmega(w_{n}(e_n))} =\smallO(1).
 \end{align*}
 
 It remains to show that $\tilde{\P}^{(n)}_x(\tau_{y} < \tilde{\tau}_{q}) \to 0$.
 Let $u$ denote the parent of $x$.
  Then it holds that
\begin{align*}
 \tilde{\P}^{(n)}_x(\tau_{y} < \tilde{\tau}_{q}) 
 &= \tilde{\P}^{(n)}_x(\tilde{X}_1 = y) + \tilde{\P}^{(n)}_x(\tilde{X}_1 = x)\tilde{\P}^{(n)}_x(\tau_{y} < \tilde{\tau}_{q})+
 \tilde{\P}^{(n)}_x(\tilde{X}_1 = u)\tilde{\P}^{(n)}_u(\tau_{x} < \tilde{\tau}_{q})\tilde{\P}^{(n)}_x(\tau_{y} < \tilde{\tau}_{q}) \\
 &\quad+ (d_n-1) \tilde{\P}^{(n)}_x(\tilde{X}_1 = y)\P_{y}^{(k,q)}(\tau_{x} < \tilde{\tau}_{q})\tilde{\P}^{(n)}_x(\tau_{y} < \tilde{\tau}_{q}).
\end{align*} 
This gives us that
\begin{align*}
 \tilde{\P}^{(n)}_x(\tau_{y} < \tilde{\tau}_{q}) &=
 \frac{\tilde{\P}^{(n)}_x(\tilde{X}_1 = y)}{1-\tilde{\P}^{(n)}_x(\tilde{X}_1 = x)-\tilde{\P}^{(n)}_x(\tilde{X}_1 = u)
 \tilde{\P}^{(n)}_u(\tau_{x} < \tilde{\tau}_{q})-
 (d_n-1) \tilde{\P}^{(n)}_x(\tilde{X}_1 = y)\tilde{\P}^{(n)}_y(\tau_{x} < \tilde{\tau}_{q})\tilde{\P}^{(n)}_x(\tau_{y} < \tilde{\tau}_{q})} \\
 & = \frac{w_{n}(e_n)}{q_n + (1-\tilde{\P}^{(n)}_x(\tau_{y} < \tilde{\tau}_{q}))w_{n}(x,u)  + w_{n}(e_n) 
 + w_{n}(e_n) (d_n-1)\paren*{1-\tilde{\P}^{(n)}_y(\tau_{x} < \tilde{\tau}_{q})}}. \label{eq:child_hitting}
\end{align*} 
 Since we have already shown that $\P^{(n)}_y(\tau_{x} < \tilde{\tau}_{q}) \to 0$,
 it follows that
 \begin{align*}
  \tilde{\P}^{(n)}_x(\tau_{y} < \tilde{\tau}_{q}) &\leq 
  \frac{w_{n}(e_n)}{w_{n}(e_n) + w_{n}(e_n) (d_n-1)\paren*{1-\tilde{\P}^{(n)}_y(\tau_{x} < \tilde{\tau}_{q})}} =\frac{1}{1 + (d_n-1)\paren*{1-\smallO(1)}} = \smallO(1).
 \end{align*}
 \end{proof}

The simple lemma below has been used to show \cref{eq:sum_convergence}.
\begin{lemma}\label{lem:sum_convergence}
 For each $n \in \N$ let $\ell_n \in \N$ be given and let 
 $(\alpha^{(n)}_i)_{i \in [\ell_n]}$ and $(\beta^{(n)}_i)_{i \in [\ell_n]}$ be real valued sequences of length $\ell_n$.
 Let $\mathcal{F}=\cbrac{f \in \N^{\N} \c f(n) \in [\ell_n] \text{ for all }n \in \N}$ 
 denote the set of choice functions on the collection $\cbrac{[\ell_1], [\ell_2], \ldots}$.
 Assume that for each $f \in \mathcal{F}$ it holds that $\alpha^{(n)}_{f(n)} \sim \beta^{(n)}_{f(n)}$ as $n \to \infty$.
 Then as $n \to \infty$ it holds that
 \begin{equation*}
  \sum_{i = 1}^{\ell_n}\alpha^{(n)}_i \sim \sum_{i = 1}^{\ell_n}\beta^{(n)}_i.
 \end{equation*}
\end{lemma}
\begin{proof}
 For all $\varepsilon > 0$ and each $f \in \mathcal{F}$, there exists an $N(\varepsilon,f) \in \N$ 
 such that for all $n \geq N(\varepsilon, f)$ it holds that 
 \begin{align*}
  \abs*{\tfrac{\alpha^{(n)}_{f(n)}}{\beta^{(n)}_{f(n)}} -1} < \varepsilon.
 \end{align*}
 Define the function $f^* \in \mathcal{F}$ by 
 \begin{equation*}
  f^*(n)=\mathrm{argmax}_{i \in [\ell_n]} \abs*{\alpha^{(n)}_{i} - \beta^{(n)}_{i}}.
 \end{equation*}
Then for all $\varepsilon > 0$ and all $n \geq N(\varepsilon, f^*)$ it holds that
\begin{align*}
 \abs*{\frac{\sum_{i = 1}^{\ell_n}\alpha^{(n)}_i}{\sum_{i = 1}^{\ell_n}\beta^{(n)}_i} - 1} 
 &= \frac{1}{\sum_{i = 1}^{\ell_n}\beta^{(n)}_i} \abs*{\sum_{i = 1}^{\ell_n}\alpha^{(n)}_i - \sum_{i = 1}^{\ell_n}\beta^{(n)}_i}  \leq \frac{\sum_{i = 1}^{\ell_n}\abs*{\alpha^{(n)}_i-\beta^{(n)}_i}}{\sum_{i = 1}^{\ell_n}\beta^{(n)}_i} \\
 &= \frac{\sum_{i = 1}^{\ell_n}\beta^{(n)}_i\abs*{\tfrac{\alpha^{(n)}_i}{\beta^{(n)}_i}-1}}{\sum_{i = 1}^{\ell_n}\beta^{(n)}_i} 
 \leq \abs*{\tfrac{\alpha^{(n)}_{f^*(n)}}{\beta^{(n)}_{f^*(n)}}-1} \frac{\sum_{i = 1}^{\ell_n}\beta^{(n)}_i}{\sum_{i = 1}^{\ell_n}\beta^{(n)}_i} = \abs*{\tfrac{\alpha^{(n)}_{f^*(n)}}{\beta^{(n)}_{f^*(n)}}-1}  < \varepsilon.
\end{align*} 
\end{proof}

\subsection{A two communities bottleneck graph}
\begin{proof}[Proof of \cref{prop:tcg_partition}]
\leavevmode \fbox{\Cref{eq:tcg_partition}}
 For a graph $G$ let $\NNM^{(G)}$ denote the non-normalized measure on $G$.
 Let $K_n$ be the complete graph on $n$ vertices.
 We can express the partition function of $BG_{n,m}$ in terms of the partition functions
 and the non-normalized measure of rooting events in the complete graphs $K_n$ and $K_m$. 
 
 Let $L_n$ denote the graph Laplacian of $K_n$.
 The partition function of $K_n$ is given by 
 \begin{align}
  Z_{K_n}(q)  &=q(q+n)^{n-1}. \label{eq:complete_partition}
 \end{align} 
 Let $U$ be a set of vertices of $K_n$ with $\abs{U}=r$ and write $[qI-L_n]_U$ to denote the submatrix of $qI-L_n$ otained by removing all rows and columns corresponding to vertices in $U$.
 Then non-normalized measure of the event that at least all vertices in $U$ are roots in a random rooted forest of $K_n$ is given by
 \begin{align}\notag
  \NNM^{(K_n)}(U \subseteq R_q) 
  &= q^r \det[qI-L_n]_U \\\notag
  &= q^r \det[(q+r)I-L_{n-r}] \\\notag
  &= q^r Z_{K_{n-r}}(q+r) \\
  &= q^r (q+r)(q+n)^{n-r-1}. \label{eq:complete_rooting}
 \end{align}
 For the partition function of $BG_{n,m}$, \cref{lem:graph_extension_single_edge} gives us that
 \begin{align*}
  Z_{BG_{n,m}}(q) &= Z_{K_n}(q)Z_{K_m}(q) + \tfrac{w}{q}Z_{K_n}(q)\NNM^{(K_m)}(b' \in R_q) + 
  \tfrac{w}{q}Z_{K_m}(q)\NNM^{(K_n)}(b \in R_q) \\
  &=q\paren*{q(q+n)(q+m) + w(q+1)(2q+n+m)}(q+n)^{n-2}(q+m)^{m-2}.
 \end{align*}
 
\leavevmode \\
\fbox{\Cref{eq:tcg_bridge_correlation}}

We can express $U_q(b,b')$ explicitly by using \cref{prop:inex,eq:complete_partition,eq:tcg_partition}
\begin{align}
 U_q(b,b') &= \frac{Z_{K_n}(q)Z_{K_m}(q)}{Z_{BG_{n,m}}(q)} \label{eq:tcg_correlation_complete_partition} \\
 &= \frac{q(q+n)(q+m)}{q(q+n)(q+m) + w(q+1)(2q+n+m)}. \label{eq:tcg_correlation_bridge_explicit}
\end{align}
The result of \cref{eq:tcg_bridge_correlation} follows directly from this expression.
 \leavevmode \\
 \fbox{\Cref{eq:tcg_cluster_correlation}}
 We will assume that $x$ and $x'$ both belong to the clique of size $n$, as the other case can be proven similarly.
  By \cref{lem:graph_extension_single_edge} we have that
 \begin{align*}
  \NNM^{(BG_{n,m})}(x \leftrightarrow_{\Phi_q} x') &= \NNM^{(K_n)}(x \leftrightarrow_{\Phi_q} x')Z_{K_m}(q)+
   \tfrac{w}{q}\NNM^{(K_n)}(x \leftrightarrow_{\Phi_q} x', \ b\in R_q)Z_{K_m}(q) \\
   &\quad+\tfrac{w}{q}\NNM^{(K_n)}(x \leftrightarrow_{\Phi_q} x')\NNM^{(K_m)}(b' \in R_q).
 \end{align*}
By \cref{eq:complete_partition,eq:complete_rooting} it follows that
 \begin{align} \notag
   \correlation[BG_{n,m}]{q}{x}{x'}
   &=1-\frac{\NNM^{(BG_{n,m})}(x \leftrightarrow_{\Phi_q} x')}{Z_{BG_{n,m}}(q)} \\\notag
   &=1-\frac{\NNM^{(K_n)}(x \leftrightarrow_{\Phi_q} x')Z_{K_m}(q)
   +\tfrac{w}{q}\NNM^{(K_n)}(x \leftrightarrow_{\Phi_q} x', \ b \in R_q)Z_{K_m}(q)}{\frac{Z_{K_n}(q)Z_{K_m}(q)}{\correlation[BG_{n,m}]{q}{b}{b'}}} \\
    &\quad +\frac{\tfrac{w(q+1)}{q(q+m)}\NNM^{(K_n)}(x \leftrightarrow_{\Phi_q} x')Z_{K_m}(q)}
    {\frac{Z_{K_n}(q)Z_{K_m}(q)}{\correlation[BG_{n,m}]{q}{b}{b'}}} \\
    &=1-\correlation[BG_{n,m}]{q}{b}{b'}
    \paren*{\paren*{1+\tfrac{w(q+1)}{q(q+m)}}\P^{(K_n)}(x \leftrightarrow_{\Phi_q} x')+
    \tfrac{w(q+1)}{q(q+n)}\P^{(K_n)}(x \leftrightarrow_{\Phi_q} x' \mid b \in R_q)}. \label{eq:bg_correlation_cluster}
    \end{align}

Let $H$ denote the graph obtained by removing all outgoing edges of $b$ from $K_n$, while retaining the ingoing edges.
By \cref{lem:spatial_markov_property} it then holds that 
$\P^{(K_n)}(x \leftrightarrow_{\Phi_q} x' \mid b \in R_q) = \P^{(H)}(x \leftrightarrow_{\Phi_q} x')$.
Let $\P_x$ denote the law of the random walk on $H$ starting at $x$ and $\tau_q$ an independent exponential killing time with rate $q$.
Since the hitting time $\tau_b$ has an exponential distribution with rate $1$, 
we can identify the random walk on $H$ killed at rate $q$ with a random walk on $K_{n-1}$ killed at rate $q+1$, 
by killing the random walk when it hits $b$. 
By analyzing Wilson's algorithm on $H$ with the first two random walks starting at $x$ and $x'$,
this gives us that
\begin{align}\notag
\P^{(K_n)}(x \leftrightarrow_{\Phi_q} x' \mid b \in R_q) &= \P^{(H)}(x \leftrightarrow_{\Phi_q} x') \\\notag
&= \P^{(K_{n-1})}(x \leftrightarrow_{\Phi_{q+1}} x') + 
 \P^{(K_{n-1})}(x \nleftrightarrow_{\Phi_{q+1}} x') \ \P_x(\tau_{b} < \tau_{q}) \ \P_{x'}(\tau_{b} < \tau_{q}) \\\notag
 &= \P^{(K_{n-1})}(x \leftrightarrow_{\Phi_{q+1}} x') + \frac{1}{(q+1)^2} \P^{(K_{n-1})}(x \nleftrightarrow_{\Phi_{q+1}} x') \\
 &= \frac{1}{(q+1)^2}+\frac{q(q+2)}{(q+1)^2} \P^{(K_{n-1})}(x \leftrightarrow_{\Phi_{q+1}} x'). \label{eq:complete_conditional_connection}
\end{align}
By \cite[Theorem 1]{avena2019meanfield} we have that 
\begin{equation}
 \P^{(K_{n})}(x \leftrightarrow_{\Phi_q} x') \to 
 \begin{cases}
  1 \text{ if } q = \smallO(\sqrt{n}) \\
  0 \text{ if } q = \smallOmega(\sqrt{n}),
 \end{cases} \label{eq:AMQ}
\end{equation}
 which together with \cref{eq:complete_conditional_connection} gives us that 
\begin{equation*}
 \P^{(K_n)}(x \leftrightarrow_{\Phi_q} x' \mid b \in R_q) \to 
 \begin{cases}
  1 \text{ if } q = \smallO(\sqrt{n}) \\
  0 \text{ if } q = \smallOmega(\sqrt{n}).
 \end{cases}
\end{equation*}

Assume that $q = \smallO(\sqrt{n})$. Fix a small enough $\varepsilon>0$.
Then for $n$ large enough it holds that $\P^{(K_n)}(x \leftrightarrow x') > 1- \varepsilon$ and that
$\P^{(K_n)}(x \leftrightarrow x' \mid b \in R_q) > 1- \varepsilon$.
By \cref{eq:tcg_correlation_bridge_explicit,eq:bg_correlation_cluster}, this means that for $n$ large enough
\begin{equation}\label{eq:tcg_correlation_both_non_bridge_vertex}
   U^{(BG_{n,m})}_{q}(x,x')
     < 1-U^{(BG_{n,m})}_{q}(b,b')
    \paren*{\paren*{1+\tfrac{w(q+1)}{q(q+m)}}(1-\varepsilon)+
    \tfrac{w(q+1)}{q(q+n)}(1-\varepsilon)} 
    =\varepsilon.
\end{equation}
If instead $q = \smallOmega(\sqrt{n})$, then analogously we find for large enough $n$ that 
\begin{align*}
   U^{(BG_{n,m})}_{q}(x,x') > 1- \varepsilon. 
\end{align*}

\leavevmode \\
\fbox{\Cref{eq:tcg_bridge_cluster_correlation}}
Assume that $x,x'$ and $b$ belong to the clique of size $n$.
By again considering the random walk on $H$, we find that
\begin{equation*}
 \P^{(K_n)}(x \leftrightarrow b \mid b \in R_q) = \P_x(\tau_b < \tau_{q}) = \frac{1}{q+1}.
\end{equation*}
So, since $\frac{1}{q+1} \to 0$ for $q=\smallOmega(\sqrt{n})$, the case $q=\smallOmega(\sqrt{n})$ follows analogous to \cref{eq:tcg_correlation_both_non_bridge_vertex}.

Now assume that $q=\smallO(\sqrt{n})$. Then we have that $\P^{(K_n)}(x \leftrightarrow_{\Phi_q} b) \to 1$, so that
\begin{align*}
  U^{(BG_{n,m})}_{q}(x,b)
  &=1-U^{(BG_{n,m})}_{q}(b,b')
    \paren*{\paren*{1+\tfrac{w(q+1)}{q(q+m)}}\P^{(K_n)}(x \leftrightarrow_{\Phi_q} b)+
    \tfrac{w(q+1)}{q(q+n)} \ \tfrac{1}{q+1}} \\
  &\sim 1-U^{(BG_{n,m})}_{q}(b,b')
    \paren*{\paren*{1+\tfrac{w(q+1)}{q(q+m)}}+
    \tfrac{w(q+1)}{q(q+n)} \ \tfrac{1}{q+1}} \\
  & = \frac{wq(q+m)}{q(q+n)(q+m) + w(q+1)(2q+n+m)}.
\end{align*}
This asymptotic expression for $U^{(BG_{n,m})}_{q}(x,b)$ gives us that
\begin{equation*}
 U^{(BG_{n,m})}_{q}(x,b)\to
 \begin{cases}
   0 & \text{ if }q=\smallO(1) \text{ or }(q=\smallO(\sqrt{n}), \ w=\smallO(m)) \text { or }(q=\smallO(\sqrt{n}), \ m=\smallO(n)) \\
   \frac{c}{1+c} & \text{ if }q=\smallOmega(1), \ q=\smallO(\sqrt{n}), \ w=\smallOmega(m), \ m\sim cn \text{ with }c \in (0,1] \\
   1 & \text{ if }q=\smallOmega(\sqrt{n})
 \end{cases}
\end{equation*}
Performing the same computation for $U_q(y,b')$ yields the result of \cref{eq:tcg_bridge_cluster_correlation}.

\leavevmode \\
\fbox{\Cref{eq:tcg_different_cluster_correlation}}
By \cref{lem:edge_probability,eq:tcg_correlation_complete_partition,eq:complete_conditional_connection,eq:tcg_correlation_complete_partition}
it holds that
\begin{align*}
 U_q^{(BG_{n,m})}(x,y) &= 1-\frac{\NNM^{(BG_{n,m})}(x \leftrightarrow_{\Phi_q} y, \ (b,b') \in \Phi_q)}{Z_{BG_{n,m}}(q)}
-\frac{\NNM^{(BG_{n,m})}(x \leftrightarrow_{\Phi_q} y, \ (b',b) \in \Phi_q)}{Z_{BG_{n,m}}(q)} \\
 &= 1-\frac{\tfrac{w}{q}\NNM^{(K_n)}(x \leftrightarrow_{\Phi_q} b, \ b \in R_q) \ \NNM^{(K_m)}(b' \leftrightarrow_{\Phi_q} y)}{Z_{BG_{n,m}}(q)}
-\frac{\tfrac{w}{q}\NNM^{(K_n)}(x \leftrightarrow_{\Phi_q} b) \ \NNM^{(K_m)}(b' \leftrightarrow_{\Phi_q} y, \ b' \in R_q)}{Z_{BG_{n,m}}(q)} \\
&= 1-U_q^{(BG_{n,m})}(b,b')
\left(
\tfrac{w}{q}\P^{(K_n)}(x \leftrightarrow_{\Phi_q} b, \ b \in R_q) \ \P^{(K_m)}(b' \leftrightarrow_{\Phi_q} y) \right. \notag \\
&\quad+\left.
\tfrac{w}{q}\P^{(K_n)}(x \leftrightarrow_{\Phi_q} b) \ \P^{(K_m)}(b' \leftrightarrow_{\Phi_q} y, \ b' \in R_q)
\right)\\
&= 1-U_q^{(BG_{n,m})}(b,b')
\paren*{\frac{w}{q(q+n)} \P^{(K_m)}(b' \leftrightarrow_{\Phi_q} y)+\frac{w}{q(q+m)}\P^{(K_n)}(x \leftrightarrow_{\Phi_q} b)} \\
&= 1-\frac{w(q+m)\P^{(K_m)}(b' \leftrightarrow_{\Phi_q} y)+w(q+n)\P^{(K_n)}(x \leftrightarrow_{\Phi_q} b)}{q(q+n)(q+m) + w(q+1)(2q+n+m)},
\end{align*}
from which the limits in \cref{eq:tcg_different_cluster_correlation} follow.
\end{proof}

\section*{{Acknowledgments}} 
{L. Avena was supported by NWO Gravitation Grant 024.002.003-NETWORKS. Parts of this work were originally initiated in the Bachelor and Master theses of J.E.P. Driessen~\cite{driessen2019lep} and V.T. Koperberg~\cite{koperberg2020lep}, respectively.}


\printbibliography

\end{document}